\documentclass[twoside,11pt]{article}

\usepackage{blindtext}

\usepackage[preprint]{jmlr2e}

\usepackage{xcolor}
\usepackage{nicefrac}
\usepackage{microtype}
\usepackage{bbm,bm}
\usepackage{enumitem}
\usepackage{amsmath,amssymb,mathtools,amsfonts}
\usepackage{thmtools,thm-restate}
\usepackage[capitalise,noabbrev]{cleveref}
\usepackage{algorithm}
\usepackage{algpseudocode}
\usepackage{parskip}
\usepackage[colorinlistoftodos,bordercolor=orange,backgroundcolor=orange!20,linecolor=orange]{todonotes}

\usepackage{caption}
\usepackage{hyperref}
\usepackage{graphicx}
\usepackage{minitoc}
\usepackage{subfigure}
\usepackage{wrapfig}
\usepackage{float}
\usepackage{multirow}
\usepackage{wrapfig,lipsum,booktabs}
\usepackage{gensymb}

\usepackage{xparse}
\usepackage{amsmath,amsfonts,bm}
\usepackage{mathtools}

\def\Figref#1{Figure~\ref{#1}}

\def\eqref#1{equation~\ref{#1}}
\def\Eqref#1{Equation~\ref{#1}}

\def\Algref#1{Algorithm~\ref{#1}}

\def\1{\bm{1}}

\DeclareMathAlphabet{\mathsfit}{\encodingdefault}{\sfdefault}{m}{sl}
\SetMathAlphabet{\mathsfit}{bold}{\encodingdefault}{\sfdefault}{bx}{n}

\def\gN{{\mathcal{N}}}

\def\gU{{\mathcal{U}}}

\newcommand{\E}{\mathbb{E}}

\newcommand{\R}{\mathbb{R}}

\DeclareMathOperator*{\argmin}{arg\,min}

\newcommand{\norm}[1]{\left\Vert #1\right\Vert}
\newcommand{\normsq}[1]{\left\Vert#1\right\Vert^{2}}
\newcommand{\inner}[2]{\langle #1,\,#2 \rangle}

\newcommand{\abs}[1]{\left \vert#1\right \vert}
\newcommand{\paren}[1]{\left(#1\right)}
\newcommand{\brackets}[1]{\left[#1\right]}
\newcommand{\braces}[1]{\left\{#1\right\}}
\newcommand{\dotp}[1]{\left\langle {#1} \right\rangle}

\newcommand{\transpose}{^\mathsf{\scriptscriptstyle T}}

\newcommand{\indicator}[1]{\mathbbm{1}\braces{#1}}

\newcommand{\var}{\mathrm{VaR}}
\newcommand{\cvar}{\mathrm{CVaR}}

\DeclareMathOperator{\prox}{prox}       
\DeclareMathOperator{\dom}{dom}			

\usepackage[colorinlistoftodos,bordercolor=orange,backgroundcolor=orange!20,linecolor=orange]{todonotes}
\floatname{algorithm}{Algorithm}

\usepackage{pifont}

\hypersetup{
    colorlinks,
    linkcolor={blue!50!black},
    citecolor={blue!50!black},
    urlcolor={blue!80!black}
}

\newtheorem{assumption}{Assumption}
\crefname{assumption}{Assumption}{Assumptions}

\crefname{thm}{Theorem}{Theorems}
\declaretheorem[name=Lemma]{lem}
\crefname{lem}{Lemma}{Lemmas}

\crefname{prop}{Proposition}{Propositions}

\crefname{cor}{Corollary}{Corollaries}

\usepackage{lastpage}
\jmlrheading{23}{2022}{1-\pageref{LastPage}}{1/21; Revised 5/22}{9/22}{21-0000}{Author One and Author Two}

\makeatletter
\newcommand{\myref}[1]{\cref{#1}\mynameref{#1}{\csname r@#1\endcsname}}
\newcommand{\Myref}[1]{\Cref{#1}\mynameref{#1}{\csname r@#1\endcsname}}

\def\mynameref#1#2{%
  \begingroup
    \edef\@mytxt{#2}%
    \edef\@mytst{\expandafter\@thirdoffive\@mytxt}%
    \ifx\@mytst\empty\else
    \space(\nameref{#1})\fi
  \endgroup
}
\makeatother

\def\adagrad/{{AdaGrad}}
\def\amsgrad/{{AMSGrad}}
\def\adam/{{Adam}}
\def\rmsprop/{{RMSProp}}
\def\adadelta/{{AdaDelta}}
\def\pytorch/{{PyTorch}}

\def\mushrooms/{{\texttt{mushrooms}}}
\def\abalone/{{\texttt{abalone}}}
\def\usps/{{\texttt{USPS}}}

\def\polyaklojasiewicz/{Polyak-{\L}ojasiewicz (PL)}

\def\SPLplus/{SPL\raisebox{.2em}{$_+$}}
\begin{document}

\title{A Model-Based Method for Minimizing CVaR and Beyond}

\author{\name Si Yi Meng \email sm2833@cornell.edu \\
       \addr Department of Computer Science\\
       Cornell University\\
       Ithaca, NY, USA
	   \AND
       \name Robert M. Gower \email rgower@flatironinstitute.org \\
       \addr Center for Computational Mathmatics \\
       Flatiron Institute\\
       New York, NY, USA
       }

\editor{My editor}

\maketitle
\begin{abstract}
We develop a variant of the stochastic prox-linear method 
for minimizing the Conditional Value-at-Risk (CVaR) objective. 
CVaR is a risk measure focused on minimizing worst-case performance, 
defined as the average of the top quantile of the losses.
In machine learning, such a risk measure is useful to 
train more robust models. Although the stochastic subgradient method (SGM)
is a natural choice for minimizing the CVaR objective, we show that our stochastic 
prox-linear (\SPLplus/) algorithm can better exploit 
the structure of the objective, while still providing a 
convenient closed form update. Our \SPLplus/ method also adapts to the scaling of the loss function,  which allows for easier tuning. 
We then specialize 
a general convergence theorem for \SPLplus/
to our setting, and show that it allows for a wider selection 
of step sizes compared to SGM. 
We support this theoretical finding experimentally. 
\end{abstract}

\section{Introduction}
\label{sec:intro}

The most common approach to fit a model parametrized by $\theta \in \R^d$ to data, is to minimize the \emph{expected} loss over the data distribution, that is 
\begin{align}
    \label{eq:expected-risk-minimization}
    \min_{\theta\in\R^d} R_{\mathrm{ERM}}(\theta) =\E_{z\sim P} [\ell(\theta; z)].
\end{align}
But in many cases, the expected loss may not be the suitable objective to minimize. 
When robustness or safety of the model are concerned, 
the emphasis should rather be on the extreme values of the distribution
rather than the average value.
For instance, 
in distributionally robust optimization, the goal is to optimize the model 
for the worst case distribution around some fixed distribution \citep{duchi2018learning}.
In extreme risk-averse settings, such as when safety is the top priority, it is desirable to minimize 
the maximum loss within a training set \citep{shalev2016minimizing}.
These applications can all be formulated as minimizing the expectation of the 
losses that are \emph{above} some cutoff value,
\begin{align}
    \label{eq:upper-tail-risk-minimization}
    \min_{\theta\in\R^d} R_{\mathrm{CVaR}}(\theta) = \E_{z\sim P} \brackets{ \ell(\theta; z) \mid \ell(\theta;z) \geq \alpha_\beta(\theta) },
\end{align}
where $\alpha_\beta (\theta)$ is the upper $\beta$--quantile of the losses. 
For example, for $\beta =0.9$,
the problem in~\labelcref{eq:upper-tail-risk-minimization} is to minimize the expectation 
of the worst $10\%$ of the losses.

In this work, we propose a variant of the stochastic prox-linear (SPL) method pioneered by~\citet{burke1995gauss,lewis2016proximal,duchi2018stochastic} for solving~\labelcref{eq:upper-tail-risk-minimization}.  
The possibility of applying SPL to CVaR minimization was mentioned in \citet{davis2019stochastic}, 
but not explored. 
We introduce a variant of SPL called \SPLplus/, that adapts to the scaling of the loss function, which in turn allows for a default parameter setting.
We first derive a closed-form update 
for \SPLplus/, and show why it is particularly well suited for minimizing CVaR. 
We give its convergence rates for convex and Lipschitz losses by adapting existing results from \citet{davis2019stochastic}. 
Through several experiments comparing the stochastic prox-linear method to stochastic subgradient we show that SPL and \SPLplus/ are more robust to the choice of step size. We conclude
with a discussion on several future applications for minimizing CVaR  in machine learning.

\subsection{Background}
The CVaR objective was first introduced in finance as an alternative measure of risk, also known
as the expected shortfall \citep{artzner1999coherent,embrechts1999extreme}.
Many applications in finance can be formulated as CVaR minimization problems,
such as portfolio optimization \citep{krokhmal2002portfolio,mansini2007conditional}, 
insurance \citep{embrechts2013modelling} and credit risk management~\citep{andersson2001credit}.
The seminal work of \citet{rockafellar2000optimization} 
proposed a variational formulation of the CVaR objective that is 
amenable to standard optimization methods.
This formulation has since inspired considerable research in applications 
spanning machine learning and adjacent fields, such as $\nu$-SVM \citep{takeda2008nu,gotoh2016cvar},
robust decision making and MDPs \citep{chow2015risk,chow2014algorithms,chow2017risk,cardoso2019risk,sani2012risk},
influence maximization and submodular optimization \citep{maehara2015risk,ohsaka2017portfolio,wilder2018risk},
fairness \citep{williamson2019fairness}, and federated learning \citep{laguel2021superquantile}.

Though it finds many applications, the CVaR objective is typically difficult to minimize.
It is nonsmooth even when
the individual losses $\ell(\cdot; z)$ are continuously differentiable. Indeed,
if $P$ does not admit a density~\textemdash~which is the case for all empirical distributions over
training data~\textemdash~the variational objective is not everywhere differentiable.
To address this,~\citet{laguel2020first} developed subdifferential calculus for a number of
equivalent CVaR formulations and proposed minimizing a smoothed version of the dual objective.
On the other hand, several works \citep{soma2020statistical,holland2021learning} apply the
stochastic subgradient method directly to the variational formulation proposed by~\citet{rockafellar2000optimization}, which is well-defined regardless
of the distribution $P$. However, as we elaborate in~\cref{sec:sgm}, this approach is oblivious to the special
structure of the variational form of the CVaR objective.
\section{Problem setup}
\label{sec:setup}
Let $\ell(\theta;z)$ be the loss associated with the model parameters
$\theta\in \R^d$ and a measurable random variable $z(\omega)$ on some background probability space $(\Omega, \mathcal{F}, \mathbb{P})$. \\[.5pc]
\begin{minipage}{0.57\textwidth}
When $z$ follows a distribution $P$ with density $p(z)$,
the cumulative distribution function on the loss for a fixed $\theta$ is given by  $\mathbb{P}[\ell(\theta;z)\leq\alpha] = \int_{\text{\raisebox{.2em}{$\scriptstyle \ell(\theta;z)\leq \alpha $}}} p(z) \,dz$,
which we assume is everywhere continuous with respect to $\alpha$. 
Let $\beta$ be a confidence level, for instance $\beta = 0.9 $. 
The Value-at-Risk (VaR) of the model is the lowest $\alpha$ such that with probability $\beta$, 
the loss will not exceed $\alpha$.  Formally, 
\begin{align}
	\label{eq:var-defn}
	\var_\beta(\theta) \coloneqq \min\braces{\alpha\in\R\,:\,\mathbb{P}[\ell(\theta;z)\leq\alpha] \geq \beta}.
\end{align}
The Conditional Value-at-Risk (CVaR) is the expectation of the upper tail starting at $\var_\beta$, illustrated in~\Figref{fig:var-cvar}:
\begin{align}
	\label{eq:cvar-defn}
	\cvar_\beta(\theta) \coloneqq 
	\E_{z\sim P}[\ell(\theta;z) \mid \ell(\theta;z) \geq \var_\beta(\theta)].
\end{align}
\end{minipage}
\hfill
\begin{minipage}{0.43\textwidth}
	\begin{center}
	  \includegraphics[scale=0.58]{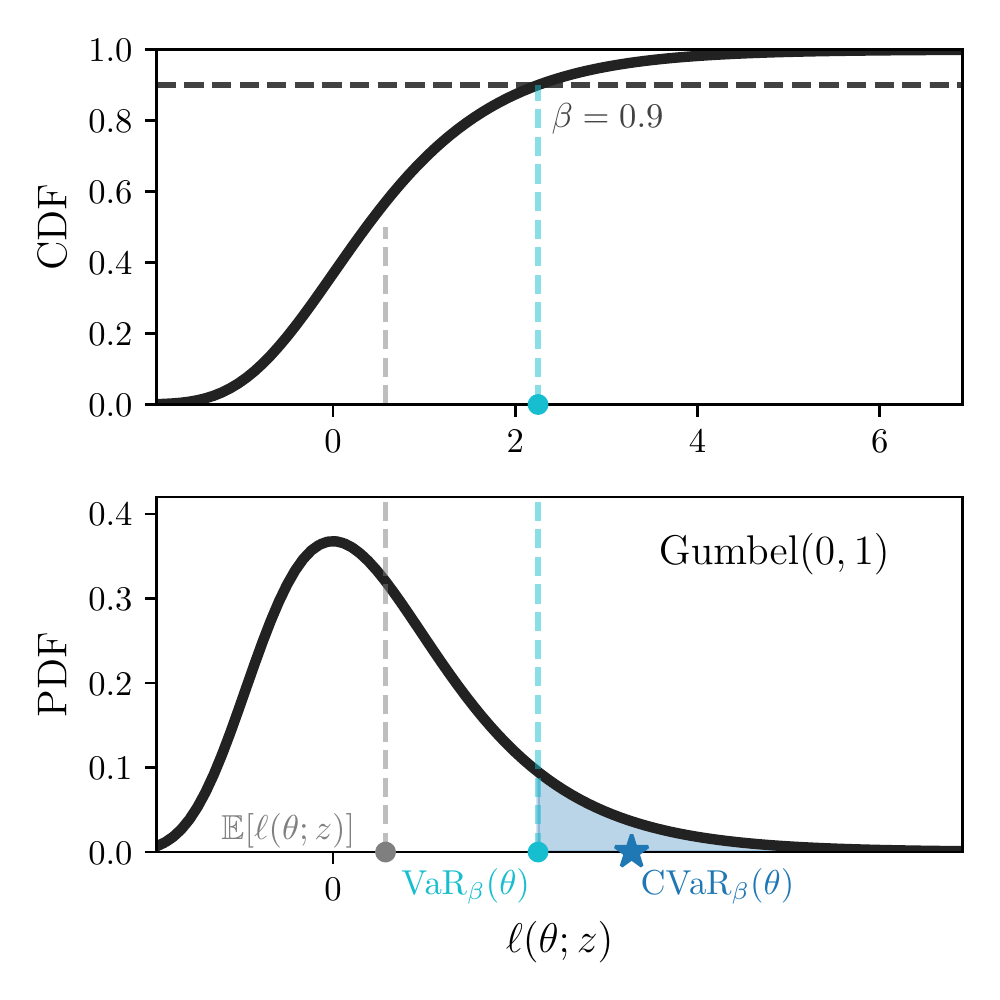}
	  	\end{center}
 	\vspace*{-1em}
    \captionof{figure}{Expectation, VaR, and CVaR.}
	\label{fig:var-cvar}
\end{minipage} 

Clearly, the CVaR upper bounds the VaR for the same $\beta$. 
Our goal is to minimize CVaR$_\beta$ over $\theta\in \R^d$,
but directly minimizing \labelcref{eq:cvar-defn} is not straightforward.
Fortunately, \citet{rockafellar2000optimization} introduced a variational formulation where the solution to
\begin{align}
	\label{eq:variational-cvar}
\theta^*,\alpha^* &\in	\argmin_{\theta\in\R^d, \alpha\in \R}  F_\beta(\theta,\alpha) \quad \mbox{where,} \\
	F_\beta(\theta,\alpha)&:= \alpha + \frac{1}{1-\beta} \E_{z\sim P} \brackets{ \max\braces{ \ell(\theta;z) - \alpha,\, 0 }} \nonumber
\end{align}
is such that $\theta^*$ is the solution to~\labelcref{eq:cvar-defn}, and 
we obtain $\alpha^* = \var_\beta(\theta)$ as a byproduct.

\section{The Stochastic Subgradient Method}
\label{sec:sgm}
A natural choice for minimizing \labelcref{eq:variational-cvar} is the stochastic subgradient method (SGM). Letting $\partial f$ denote the convex subdifferential
of $f$, at each step $t$ we sample $z\sim P$ uniformly and compute a subgradient $g_t$ from the subdifferential
\begin{align}
	\label{eq:cvar-sampled-subgrad}
	\partial F_\beta(\theta_t,\alpha_t;  z)  &=
   \begin{pmatrix}
        \mathbf{0} \\ 1
    \end{pmatrix}  
      + 
    \frac{u_t}{1-\beta}  \begin{pmatrix}
        \partial \ell(\theta_t;z) \\ -1
    \end{pmatrix} 
\end{align}
where $u_t  = \left. \partial \max\{u, 0\}\right|_{\text{\raisebox{.2em}{$u = \ell(\theta_t; z) - \alpha_t$}}}$. 
Given some step size sequence  $\braces{\lambda_t}>0$, and denoting $x = (\theta,\alpha)\transpose$, SGM then takes the step 
\begin{equation}\label{eq:SGM}
 x_{t+1} = x_t - \lambda_t g_t, \quad  \mbox{where } g_t\in\partial F_\beta(\theta_t,\alpha_t;z).    
\end{equation}
Substituting in the subgradient $g_t$ given in~\labelcref{eq:cvar-sampled-subgrad} into~\labelcref{eq:SGM} gives
\begin{align}\label{eq:SGM2}
 \theta_{t+1} &= \theta_t - \frac{\lambda_t}{1-\beta} u_t \partial \ell(\theta_t;z),  \\
 \alpha_{t+1} & = \alpha_t-\lambda_t + \frac{\lambda_t}{1-\beta} u_t, 
\end{align} 
For reference, the complete SGM algorithm is given in \Algref{alg:sgd-cvar}.
SGM is very sensitive to the step size choice and 
may diverge if not carefully tuned. 
\begin{algorithm}[h]
    \begin{algorithmic}[1]
    \State \textbf{initialize:}  $\theta_0\in \R^{d}$, $\alpha_0 \in \R,$ \textbf{hyperparameter:} $\lambda>0$ 
    \For {$t = 0, 1, 2, \dots, T$}
        \State Sample data point $z\sim P$, compute $\ell(\theta_t;z)$ and $v_t\in \partial \ell(\theta_t;z)$
		\State $\lambda_t\gets \lambda / \sqrt{t+1}$
        \If{ $\alpha_t \geq \ell(\theta_t;z) $ } \Comment{$\alpha_t$ too big}
            \State $\theta_{t+1} \gets \theta_t$
            \State $\alpha_{t+1} \gets \alpha_t - \lambda_t$
        \Else \Comment{$\alpha_t$ too small}
            \State $\theta_{t+1} \gets \theta_t - \frac{\lambda_t}{1-\beta} v_t $
            \State $\alpha_{t+1} \gets \alpha_t + \frac{\lambda_t}{1-\beta} \beta$
        \EndIf
    \EndFor
    \State \Return $\bar x_T = \frac{1}{T+1} \sum_{t=1}^{T+1} (\theta_t, \alpha_t)\transpose$
    \end{algorithmic}
    \caption{SGM: Stochastic subgradient method for CVaR minimization }
    \label{alg:sgd-cvar}
\end{algorithm}
This issue can be explained from a 
modeling perspective \citep{davis2019stochastic}.
Indeed, SGM can be written as a model-based method where at each iteration $t$,
it uses the following linearization of the sampled $F_\beta(x;z)$ at the current point $x_t$:
\begin{align}
	\label{eq:sgm-model}
	m_{t}^{\mathrm{SGM}}(x;z) := F_\beta(x_t;z) + \inner{g_t}{x - x_t}.
\end{align}
This provides an approximate, stochastic model of the objective $F_\beta(x)$. 
The SGM update is then a proximal step on this model, that is
\begin{align}
	\label{eq:stoch-model-based-update}
	x_{t+1} = \argmin_{x\in\R^{d+1}} \,  m_{t}(x;z) + \frac{1}{2\lambda_t}\normsq{x - x_t}
\end{align}
using $m_t = m_{t}^{\mathrm{SGM}}$. The issue with $m_t=m_{t}^{\mathrm{SGM}}(x;z)$ is that it uses a linearization to approximate the 
$\max\{\cdot,\,0\}$ function. 
This linearization can take negative values, which is a poor approximation of the non-negative $\max\{\cdot,\, 0\}$ operation.
The main insight of the SPL method is to leverage the structure of
$F_\beta(x)$ as a truncated function.
This structure allows for a more accurate model that still has an easily computable proximal operator. 

\section{The SPL method for CVaR minimization}
\label{sec:algorithm}

\subsection{A tighter model}
Here we introduce an alternative model for our objective that only linearizes \emph{inside} the $\max\{\cdot,0\}$,
which is a strictly more accurate model when the objective is convex~\citep{asi2019importance}.
In particular, for some $v_t\in\partial\ell(\theta_t;z)$ and $\ell_t := \ell(\theta_t;z)$, we use
\begin{align}
	\label{eq:spl-model}
	m_t^{\mathrm{SPL}}(\theta,\alpha;z) & = \alpha + \frac{1}{1-\beta} \max \braces{ \ell_t + \inner{ v_t }{ \theta - \theta_t } - \alpha, \, 0}
\end{align}

\begin{figure}
	\vspace*{-1em}
	\begin{center}
		\includegraphics[scale=0.6]{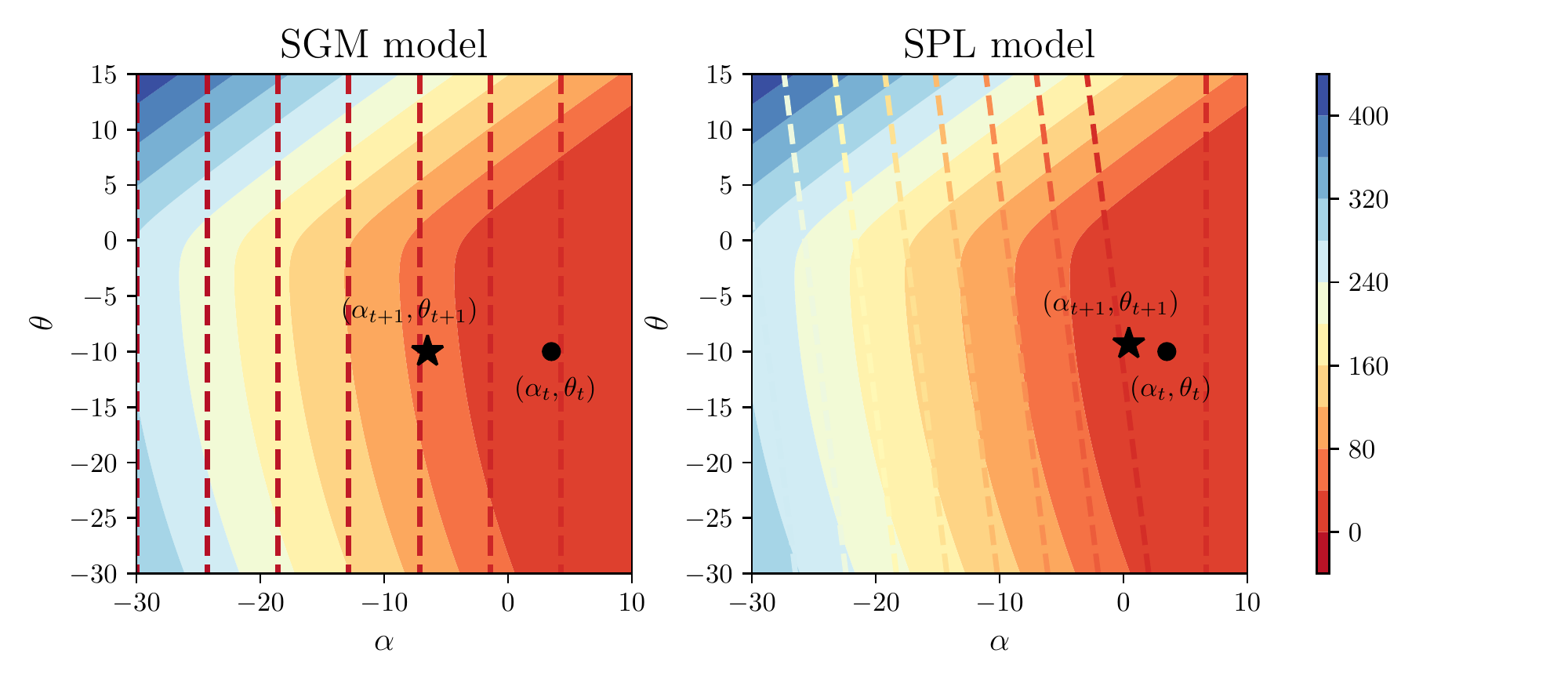}
	\end{center}
	\vspace*{-1em}
  	\caption{Comparison of SGM and SPL models on the CVaR objective with a single $\ell(\theta) = \log(1+\exp(\theta)) + \frac{0.01}{2}\theta^2$. Filled contours are the level sets of the objective, while the dashed contour lines are the level sets of the respective model $m_t$ constructed at $(\theta_t,\alpha_t)$.
	With the same step size, the SGM model results in an update that increases the objective, whereas the SPL model does not.
	Note that because the subgradient of the objective is $0$ in $\theta$, the SGM model is constant in $\theta$.
	}
	\label{fig:cvar-model-compare}
\end{figure}%
The algorithm resulting from \labelcref{eq:stoch-model-based-update} using $m_t = m_t^{\mathrm{SPL}}$
is known as the stochastic prox-linear (SPL) method \citep{duchi2018stochastic}.
\Figref{fig:cvar-model-compare} illustrates that~\labelcref{eq:spl-model} better approximates the level sets of the loss function as compared to~\labelcref{eq:sgm-model}.

\subsection{Separate regularization parameters}
Now that we have determined a tighter model~(\ref{eq:spl-model}), it remains now to select a default step size sequence $\lambda_t$ for the proximal step~\labelcref{eq:stoch-model-based-update}.   But, as we will argue next, having the same default step size sequence for both $\alpha$ and $\theta$ could lead to inconsistencies
due to the dependency on the \emph{scale} of the loss function. 

To explain this dependency, let units$(\ell)$ denote the \emph{units} of our loss function $\ell(\theta_t;z).$ For instance, our loss could be a cost measured in dollars. 
Since $\alpha$ approximates a quantile of the losses, 
it must also have the same units as the loss. Consequently, our model in \labelcref{eq:spl-model} also has the same units as the loss function. 
A clash of units appears when we consider the regularization term in~\labelcref{eq:stoch-model-based-update}, that is the term
\begin{align*}
	\frac{1}{2\lambda_t} \norm{x-x_t}^2 =\frac{1}{2\lambda_t} \left(\norm{\theta-\theta_t}^2 + (\alpha-\alpha_t)^2\right). 
\end{align*}
This regularization term must also have the same units as the loss so that the entire objective in~\labelcref{eq:stoch-model-based-update} has consistent units. But since units$(\alpha) = \,$units($\ell$), 
the term $\frac{1}{2\lambda_t}  (\alpha-\alpha_t)^2$ 
can only have the same units as the loss if units$(\lambda_t) = \,$units($\ell$). 
In direct contradiction, the term $\frac{1}{2\lambda_t} \norm{\theta-\theta_t}^2 $ can only have the same units as the loss if units$(\lambda_t) = 1/$units($\ell$),
since $\theta$ parametrizes the objective and thus does not carry the units of the loss.
There is no choice of $\lambda_t$ which would result in the objective of~\labelcref{eq:stoch-model-based-update} having consistent units; 
consequently, there is no default, scale-invariant $\lambda_t$ that would work across different loss functions.

One simple way to fix this clash of units, is to disentangle $\lambda_t$ into two regularization parameters $\lambda_{\theta,t}, \lambda_{\alpha,t}>0$ and update the iterates according to
\begin{align}
	\label{eq:stoch-model-based-update-2}
	\theta_{t+1},\alpha_{t+1} = \argmin_{\theta\in\R^{d}, \alpha \in\R} \,  &m_t^{\mathrm{SPL}}(\theta,\alpha;z) + \frac{1}{2\lambda_{\theta,t}}\normsq{\theta - \theta_t}   + \frac{1}{2\lambda_{\alpha,t}}(\alpha_t - \alpha)^2.
\end{align}
Now we can make the units match across~\labelcref{eq:stoch-model-based-update-2} by choosing 
\begin{equation}\label{eq:units-main}
     \mbox{units}(\lambda_{\alpha,t}) = \mbox{units}(\ell) \quad \mbox{and}\quad \mbox{units}(\lambda_{\theta,t}) = \frac{1}{\mbox{units}(\ell)}.
\end{equation}
As suggested by our theory in \labelcref{eq:best-bounds-lambdas},
if we had access to the average Lipschitz constant $L$ of the the individual losses $\ell$,
then we should choose 
\begin{equation}\label{eq:units-main-Lipschitz}
    \lambda_{\alpha,t} = \frac{\lambda |\alpha_t-\alpha^*|}{\sqrt{t}} \quad \mbox{and}\quad \lambda_{\theta,t} = \frac{\lambda\norm{\theta_t-\theta^*}}{L\sqrt{t}},
\end{equation}
where $\lambda >0$ is a numerical constant. 
Although this gives us consistency in the units,
estimating $L$ can be difficult in practice. 
Thus, instead we approximate the scaling 
by using the initial loss $\ell_0 := \E_{z}[\ell(\theta_0;z)]$ and choose
\begin{equation}\label{eq:units-main-loss}
    \lambda_{\alpha,t} = \frac{\lambda \ell_0}{\sqrt{t}} \quad \mbox{and}\quad \lambda_{\theta,t} = \frac{\lambda}{\ell_0 \sqrt{t}},
\end{equation}
while setting $\lambda$ using a grid search.
We will use~\labelcref{eq:units-main-loss} as our default setting for $\lambda_{\theta,t}$ and $\lambda_{\alpha,t}$. 
Importantly, although we have separate regularization terms,
there is still only one hyperparameter $\lambda$ to be set.

\subsection{Closed form update}
\begin{restatable}[Closed form updates of \SPLplus/]{lem}{closedupdate}
\label{lem:closed-form-update}
The closed form solution to~\labelcref{eq:stoch-model-based-update-2} is given by the updates
\begin{align}  
    \theta_{t+1} & =      \theta_t -\lambda_{\theta,t} \min\braces{ \frac{1}{1-\beta},\, \gamma_t }   \nabla\ell(\theta_t;z), \label{eq:theta-update-SPL-2reg}\\
    \alpha_{t+1} & = \alpha_{t} -\lambda_{\alpha,t}+ \lambda_{\alpha,t} \min\braces{ \frac{1}{1-\beta},\, \gamma_t}, \label{eq:alpha-update-SPL-2reg} \\
	\text{where }\; \gamma_t & = \frac{ \max\braces{\ell(\theta_t;z) - \alpha_t +\lambda_{\alpha,t}, \,0} }{ \lambda_{\theta,t}\normsq{ \nabla\ell(\theta_t;z)} +\lambda_{\alpha,t}}. \label{eq:gamma-update-SPL-2reg}
\end{align}
\end{restatable}
We first give a sketch of how the updates are derived.
\begin{proof}
For one step update, we can drop the subscript $t$ without loss of generality.
The key step is to rewrite \labelcref{eq:stoch-model-based-update-2}
in the form of a proximal step on a truncated model, namely,
\begin{align*}
	x_{t+1} = \argmin_{x\in\R^{d+1}} \, \max\braces{ c + \inner{a}{x-x_t},\,0 } + \frac{1}{2\lambda} \normsq{x-x_t}
\end{align*}
where $x = (\theta,\hat\alpha)^\top$ is the concatenation of $\theta$ and a scaled version of $\alpha$.
The solution to this has a nice form given in \cref{lem:truncated-model} in the appendix,
\begin{align*}
	x^{t+1} = x_t -  \underbrace{ \min\braces{\lambda,\, \frac{ \max\braces{c,\,0} }{\normsq{a}}}}_{\smash{\eqqcolon \eta}} a.
\end{align*} 
One can show that by redefining variables as 
\begin{align*}	
	\hat\alpha = \sqrt{\frac{\lambda_\theta}{\lambda_\alpha}}\alpha  \qquad \text{and} \qquad \hat \alpha_t = \sqrt{\frac{\lambda_\theta}{\lambda_\alpha}} \alpha_t - \sqrt{\lambda_\theta\lambda_\alpha},
\end{align*}
we can absorb the leading $\alpha$ in the model~(\ref{eq:spl-model}) into its regularization term, giving us 
\begin{align*}
	\alpha + \frac{1}{2\lambda_\alpha}(\alpha-\alpha_t)^2 = \frac{1}{2\lambda_\theta} (\hat\alpha - \hat\alpha_t)^2 + \text{Const. }.
\end{align*} 
After some simple manipulation on the linearization term of \labelcref{eq:stoch-model-based-update-2}, we get that 
\begin{align*}
	c = \frac{1}{1-\beta} \paren{ \ell(\theta_t;z) - \sqrt{\frac{\lambda_\alpha}{\lambda_\theta}}\hat\alpha_t}, \quad a = \frac{1}{1-\beta}\begin{pmatrix}
		\nabla\ell(\theta_t;z) \\
		-\sqrt{\frac{\lambda_\alpha}{\lambda_\theta}}.
	  \end{pmatrix}.
\end{align*} 
Plugging $a,c$ into the update of the truncated model above, 
\begin{align*}
	\eta = \min\braces{ \lambda_\theta,\, \frac{ \max\braces{\ell(\theta_t;z) - \sqrt{\frac{\lambda_\alpha}{\lambda_\theta}}\hat\alpha_t, \,0} }{\frac{1}{(1-\beta)} (\normsq{ \nabla\ell(\theta_t;z)} +\frac{\lambda_{\alpha}}{\lambda_{\theta}} ) } } .
\end{align*}
Substituting out $\hat\alpha_t$ for $\alpha_t$
and multiplying by $a$ gives us the desired $\theta_{t+1}$ and $\alpha_{t+1}$.
\end{proof}
The detailed proof can be found in \cref{sec:app-alg}, with a breakdown of the updates in~\Algref{alg:spl-cvar-two-reg}. 
Alternative to our technique, one can also derive these updates by enumerating the KKT conditions 
after formulating \labelcref{eq:stoch-model-based-update-2} as a constrained minimization problem with an additional slack variable.

Examining the update in Lemma~\ref{lem:closed-form-update}, we can see that the cost of computing each iteration of \SPLplus/ is of the same order as computing an iteration of SGM. Finally, if we set the regularization parameters according to the guide in~\labelcref{eq:units-main}, we can see by examining the units of \SPLplus/ that $\gamma_t$ in~\labelcref{eq:gamma-update-SPL-2reg} is \emph{unitless}. 
As a result, the units are consistent across the updates of both $\theta$ in~\labelcref{eq:theta-update-SPL-2reg} and $\alpha$ in~\labelcref{eq:alpha-update-SPL-2reg}.
Next, we discuss two applications of \SPLplus/ which correspond to two extreme settings for the CVaR objective.

\begin{algorithm}[h]
    \begin{algorithmic}[1]
    \State \textbf{initialize:}  $\theta_0\in \R^{d}$, $\alpha_0 \in \R,$ \textbf{hyperparameter:} $\lambda>0$ 
    \For {$t = 0, 1, 2, \dots, T$}
        \State Sample data point $z\sim P$
        \State Compute $\ell(\theta_t;z)$ and $v_t\in \partial \ell(\theta_t;z)$
		\State $\lambda_{\theta,t}\gets \lambda  / (\ell_0 \sqrt{t+1})$
		\State $\lambda_{\alpha,t}\gets \lambda  \ell_0 / \sqrt{t+1}$ 
        \If{ $\alpha_t > \ell(\theta_t;z) + \lambda_{\alpha,t} $ } \Comment{$\alpha_t $ too big}
            \State $\theta_{t+1} \gets \theta_t$
            \State $\alpha_{t+1} \gets \alpha_t - \lambda_{\alpha,t} $
        \ElsIf{ $\alpha_t < \ell(\theta_t;z) -  \frac{\lambda_{\theta,t}}{1-\beta}\normsq{v_t} - \frac{\lambda_{\alpha,t} \beta}{1-\beta}$ }  \Comment{$\alpha_t$ too small}
            \State $\theta_{t+1} \gets \theta_t - \frac{\lambda_{\theta,t}}{1-\beta} v_t $
            \State $\alpha_{t+1} \gets \alpha_t +  \frac{\lambda_{\alpha,t}}{1-\beta}\beta$
        \Else \Comment{$\alpha_t$ in middle range}
\State $\nu\gets \frac{\ell(\theta_t;z) + \lambda_{\alpha,t} - \alpha_t}{\lambda_{\theta,t} \normsq{v_t} + \lambda_{\alpha,t} }$
\State $\theta_{t+1} \gets \theta_t - \lambda_{\theta,t}\nu \nabla \ell(\theta_t; z) $
\State $\alpha_{t+1} \gets \alpha_t - \lambda_{\alpha,t} + \lambda_{\alpha,t}\nu $
        \EndIf
    \EndFor
	\State \Return $\bar x_T = \frac{1}{T+1} \sum_{t=1}^{T+1} (\theta_t, \alpha_t)\transpose$
    \end{algorithmic}
    \caption{\SPLplus/: Stochastic prox-linear method for CVaR minimization with separate regularization }
    \label{alg:spl-cvar-two-reg}
\end{algorithm}

\subsection{Solving the max loss problem}
The \SPLplus/ method can been seen as an extension of recent class of adaptive methods~\citep{slackpolyak} for minimizing the max loss, as we detail next.
If $P$ is the empirical distribution over $n$ training examples, 
setting $\beta= \nicefrac{n-1}{n}$ turns the CVaR minimization problem into the 
max loss minimization problem 
\begin{align}
	\label{eq:max-loss-minimization-problem}
	\min_{\theta\in\R^d} \; f(\theta) = \max_{i=1,\dots,n} \ell(\theta;z_i).
\end{align}
Indeed, if $\beta= \nicefrac{n-1}{n}$ then 
the Value-at-Risk \labelcref{eq:var-defn} would have to be the max loss, that is, 
$\alpha =\max_{i=1,\dots,n} \ell(\theta;z_i)$. 
Plugging this into~(\ref{eq:variational-cvar}) we have that the second term in $F_\beta(\theta,\alpha)$ is zero, 
leaving only $F_\beta(\theta,\alpha) =\alpha = \max_{i=1,\dots,n} \ell(\theta;z_i).$

The max loss problem is an interesting problem in its own right~\citep{shalev2016minimizing}. 
Recently \citet{slackpolyak} proposed the \emph{Polyak with slack} methods for solving~\labelcref{eq:max-loss-minimization-problem}. Our \SPLplus/ improves upon the Polyak with slack methods in two ways: 
first, \SPLplus/ can be applied to minimizing CVaR for any $\beta,$ and not just the max loss problem; 
second, \SPLplus/ can enjoy a default parameter setting due to the two regularization parameters and the consideration around units in~\labelcref{eq:units-main}.

Finally we show that in this setting \SPLplus/ can also been seen as a stochastic algorithm that minimizes 
the Lagrangian of a slack formulation of \cref{eq:max-loss-minimization-problem},
where the Lagrange multiplier is equal to $\nicefrac{1}{1-\beta}$. We establish this equivalence in 
\cref{sec:relationship-to-max-loss-minimization}.

\subsection{Solving ERM}
When $P$ is the empirical distribution over $n$ training examples, and if
 $\beta =  \frac{1}{n}$,  then minimizing the CVaR objective in \labelcref{eq:variational-cvar}
is equivalent to minimizing the expected risk.  This is because $\alpha = \min_{i=1,\ldots, n} \ell(\theta,z_i)$ due to~(\ref{eq:var-defn}), and consequently from~(\ref{eq:cvar-defn}) we have that 
\begin{align*}
     \mbox{CVaR}_\beta(\theta) &= 
\E_{z\sim P}[\ell(\theta;z) \mid \ell(\theta;z) \geq \min_{i=1,\ldots, n} \ell(\theta,z_i)]   \\
&=\E_{z\sim P}[\ell(\theta;z)]. 
\end{align*}
Thus minimizing~(\ref{eq:variational-cvar}) is equivalent to minimizing the expected risk. 
As a consequence,  \SPLplus/  can also be used as an adaptive method for minimizing the expected risk. 

\section{Convergence theory}
We instantiate the convergence analyses from \citet{davis2019stochastic} 
in the case of CVaR minimization, and compare the rates for SGM and \SPLplus/ 
for losses satisfying the following Assumption.

\begin{assumption}[Convex, subdifferentiable, and Lipschitz]
	\label{asmpt:convex-diff-lipschitz}
	There exist square integrable random variables $M:\Omega\to\R$ such that  
	for a.e. $z\in\Omega$ and all $\theta\in\R^d$, the sample losses $\ell(\theta;z)$ are convex, 
	subdifferentiable\footnote{Historically, the prox-linear method was proposed for composite optimization problems where the inner function is $C^1$~\cite{BF95}. Here we slightly abuse the terminology and allow for general subdifferentiable losses $\ell(\cdot;z)$.}, and $M(z)$-Lipschitz.
\end{assumption}
\begin{restatable}[Convergence rates of SGM and \SPLplus/]{theorem}{convergenceconvex}
	\label{thm:general-convex-rate}
	Suppose Assumption~\ref{asmpt:convex-diff-lipschitz} holds. Let $x^* = (\theta^*,\alpha^*)\transpose$ be a minimizer of $F_\beta(\theta,\alpha)$,
	and $x_0 \in \R^d$ an arbitrary initialization. Let $(x_t)_{t=0}^{T}$ be the iterates given by 
	SGM or \SPLplus/, and $\bar x_T = \tfrac{1}{T+1}\sum_{t=1}^{T+1} x_t$  be the averaged iterate.
 
{\bf SGM.} If  $\lambda_t = \frac{\lambda}{\sqrt{T+1}}$ then 
the iterates $(x_t)$ given by SGM in~\labelcref{eq:SGM} satisfy
\begin{align}\label{eq:conv2paramSGD}
  \E\brackets{  F_\beta(\bar x_T) - F_\beta( x^* ) } \leq\frac{1}{2}\frac{\norm{\theta_0-\theta^*}^2}{\lambda\sqrt{T+1}} 
  + \frac{1}{2}\frac{( \alpha_0 - \alpha^*)^2}{\lambda\sqrt{T+1}} + \frac{\lambda\mathsf{L}_{\mathrm{SGM}}^2  }{\sqrt{T+1}}, 
 \end{align}
 where  
\begin{align}\label{eq:Lip-SGM}
\mathsf{L}_{\mathrm{SGM}}^2 &= \E_z\brackets{\frac{M(z)^2+1 }{(1-\beta)^2}+ 1  } 
\end{align}

{\bf \SPLplus/.} If  $\lambda_{\alpha,t} = \frac{\lambda_{\alpha}}{\sqrt{T+1}}$ 
 and 	$\lambda_{\theta,t} = \frac{\lambda_{\theta}}{\sqrt{T+1}}$, then the iterates $(x_t)$ given by \SPLplus/ given in~\cref{lem:closed-form-update}
  satisfy
\begin{align}\label{eq:conv2param}
  \E\brackets{  F_\beta(\bar x_T) -  F_\beta( x^* ) } \leq\frac{1}{2}\frac{\norm{\theta_0-\theta^*}^2}{\lambda_\theta\sqrt{T+1}} 
  + \frac{1}{2}\frac{( \alpha_0 - \alpha^*)^2}{\lambda_\alpha\sqrt{T+1}} + \frac{\lambda_\alpha \mathsf{L}_{\mathrm{SPL}}^2  }{\sqrt{T+1}}, 
\end{align}
	where 
\begin{align}
\label{eq:Lip-SPL}
\mathsf{L}_{\mathrm{\SPLplus/}}^2 = \E_z \brackets{ \frac{\frac{\lambda_\theta }{\lambda_\alpha} M(z)^2 + 1 }{(1-\beta)^2}   }.
\end{align}
\end{restatable}
This result follows by adapting Theorem 4.4 in \citet{davis2019stochastic}, and we verify the
assumptions necessary in \cref{sec:proofs}.
In particular, the best bound achieved by SGM via minimizing in $\lambda$ the RHS of \labelcref{eq:conv2paramSGD} 
is with 
\begin{align}
	\label{eq:best-bounds-lambda-sgm}
	\lambda = \frac{ \Delta }{\mathsf{L}_{\mathrm{SGM}}\sqrt{2}}
\end{align}
yielding the rate 
\begin{align*}
	\E \brackets{ F_\beta(\bar x_T) - F_\beta(x^*) } \leq \frac{\sqrt{2} \norm{x_0-x^*} \mathsf{L}_{\mathrm{SGM}}}{ \sqrt{T+1} }. 
\end{align*}
Similarly, for \SPLplus/, the best bound is achieved at 
\begin{align}
	\label{eq:best-bounds-lambdas}
	\lambda_\alpha = \frac{ \abs{ \alpha_0 - \alpha^* } (1-\beta) }{ \sqrt{2} } , \quad \lambda_\theta = \frac{ \norm{\theta_0 - \theta^*} (1-\beta) }{ \sqrt{2}\, \E_z [M(z)] },
\end{align}
giving us the rate 
\begin{align*}
	 \E \brackets{ F_\beta(\bar x_T) - F_\beta(x^*) } &\leq \frac{ \norm{\theta_0-\theta^*} \E_z[M(z)] + \abs{\alpha_0 - \alpha^*} }{ \sqrt{2}(1-\beta)\sqrt{T+1} } \\
	& \qquad +  \frac{ \norm{\theta_0 - \theta^*} \E_z[M(z)^2] / \E_z[M(z)] + \abs{\alpha_0 - \alpha^*} }{\sqrt{2} (1-\beta) \sqrt{T+1} }.
\end{align*}
We can now use \cref{thm:general-convex-rate} to directly compare the convergence rate of SGM in~(\ref{eq:conv2paramSGD}) and \SPLplus/ in~(\ref{eq:conv2param}). First, both methods converge at the $O\paren{\nicefrac{1}{\sqrt{T+1}}}$ rate. 
The main difference is in the constants. To ease the comparison, let $\lambda_\alpha = \lambda_\theta = \lambda$. In this case, we can see that the Lipschitz constant of SGM in~(\ref{eq:Lip-SGM}) is always greater than the Lipschitz constant of \SPLplus/ in~(\ref{eq:Lip-SPL}), thus \SPLplus/ has a better constant in its rate of convergence. This is another way to confirm that \SPLplus/ uses a better model of the objective function as compared to SGM. 
Yet another advantage of \SPLplus/ is the flexibility of having two regularization parameters $\lambda_{\theta}$ and $\lambda_{\alpha}$, which allows for a method that is independent of the units of the loss.

\section{Experiments}
\label{sec:experiments}
We design several experiments to compare, and test the sensitivity of 
SGM, SPL with only one regularization, that is the updates in Lemma~\ref{lem:closed-form-update} where $\lambda_{\theta,t}=\lambda_{\alpha,t} = \lambda_t$, 
and our proposed \SPLplus/ updates. 

\subsection{Synthetic data}
\label{sec:exp-syn}
First we study the sensitivity of the methods to choices of $\lambda$ 
when minimizing the CVaR objective \labelcref{eq:variational-cvar}.
We use three different synthetic distributions,
similar to the setup of~\citet{holland2021learning}, 
where we experiment various combinations of loss functions $\ell(\cdot; z)$
and data distributions controlled by noise $\zeta$ (Table~\ref{tab:exp-losses}).
For all problems we set the dimension to be $d=10$. 
For regression problems, $\theta_{\mathrm{gen}}\sim\gU([0,1]^d)$,\\[.5pc]
\begin{minipage}{0.57\textwidth}
and for classification (logistic regression) we use $\theta_{\mathrm{gen}}\sim\gU([0,10]^d)$
to increase linear separability. The loss functions and target generation schemes are listed in \cref{tab:exp-losses}.
Each target of the corresponding problem contains an error $\epsilon$ from one of the distributions in
\cref{tab:exp-error-distrs}, which controls the difficulty level of the problem.
\end{minipage}
\hfill
\begin{minipage}{0.39\textwidth}
		\centering
		 \begin{tabular}{@{}ll@{}}
				\toprule
				 \textbf{Distribution of $\zeta$}  & \textbf{Parameters} \\ \midrule
				 $\mathrm{Normal}(\mu,\sigma^2)$ & $\mu=0$, $\sigma = 2$   \\
				 $\mathrm{Gumbel}(\mu,\beta)$ & $\mu=0$, $\beta=4$   \\
				 $\mathrm{LogNormal}(\mu, \sigma^2)$ &  $\mu=2$, $\sigma = 1$ \\ \bottomrule
				\end{tabular}
	   \captionof{table}{Error distributions in 1D.}
	   \label{tab:exp-error-distrs}
\end{minipage} 

\begin{table*}
	\caption{Loss functions and data generation used for synthetic problems.
	The error distributions for $\zeta$ are described in \cref{tab:exp-error-distrs}.
	We use $\sigma(\cdot)$ to denote the sigmoid function, and all $x$'s are sampled uniformly from the unit sphere.}
	\label{tab:exp-losses}
	\centering \begin{tabular}{@{}lll@{}}
	\toprule
	 \textbf{Task} & \textbf{Loss} $\ell(\theta;x,y)$  &
  \textbf{Target}  \\ \midrule
	 Regression & $ \frac{1}{2}(x\transpose \theta - y)^2$  &  $y = x\transpose \theta_{\mathrm{gen}} + \zeta$ \\
	 Regression & $ \abs{x\transpose \theta - y}$  & $y = x\transpose\theta_{\mathrm{gen}} + \zeta$ \\
	 Classification & $ \log\paren{ 1 + \exp\paren{ -yx\transpose\theta }}$ & $y = 1$ w.p. $\sigma(x\transpose\theta_{\mathrm{gen}} + \zeta)$ and $-1$ otherwise. \\ \bottomrule
	\end{tabular}
\end{table*}

Since the expectation in the CVaR objective \labelcref{eq:variational-cvar} is difficult to compute in closed form, we
evaluate the suboptimality gaps using an empirical average over $N = 10^6$ data
points sampled i.i.d. from the corresponding distribution under a single fixed seed. 
This is done for each error distribution and loss function combination, 
each giving us the discretization
\begin{align}
	\label{eq:empirical-cvar}
	\tilde F_\beta(\theta,\alpha) = \alpha + \frac{1}{1-\beta}\frac{1}{N}\sum_{i=1}^N \max\braces{\ell(\theta;z_i) -\alpha ,\,0}. \,\hspace{-0.18cm}
\end{align}
We set $\beta=0.95$ for all experiments, and thus have omitted $\beta$ from all plot descriptions.
We run full-batch L-BFGS to obtain the optimal values for comparison, recorded as
$\theta^*, \alpha^*$, and $F^* \coloneqq \tilde F_\beta(\theta^*,\alpha^*)$.
For initialization, we set $\alpha_0\sim\gU(0,1)$ and $\theta_0\sim\gN(0,I_d)$
at initialization for all algorithms we compare.
They are run for $T=100,000$ iterations using 5 different seeds that control the randomness 
of initialization and sampling during
the course of optimization. In the sensitivity plots (\Cref{fig:subopt-vs-lambda,fig:iters-vs-lambda}),
solid lines show the median values, while the shaded regions indicate the range over the random seeds.
All objective evaluations are on $\tilde F_\beta(\bar\theta_t,\bar\alpha_t)$ using 
the averaged iterates.

We employ a decreasing step size $\lambda_t = \nicefrac{\lambda}{\sqrt{t+1}}$ for SGM and SPL, while
$\lambda_{t,\alpha} = \nicefrac{\lambda \ell_0}{\sqrt{t+1}}$ and $\lambda_{t,\theta} = \nicefrac{\lambda }{\ell_0 \sqrt{t+1}}$ 
for \SPLplus/.
We study the sensitivity of the methods to $\lambda$, varied
over a logarithmically-spaced grid $10^{-6}, 10^{-5}, \dots, 10^{4}$,
densified around $\lambda=1$ using the extra grid $10^{-1.5}, 10^{-0.5}, \dots, 10^{1.5}$.

\Figref{fig:subopt-vs-lambda} shows the final suboptimality achieved by SGM, SPL,
and \SPLplus/ for different values of $\lambda$. 
For smooth losses (squared and logistic) we see that SPL and \SPLplus/ are significantly 
more robust and admit a much larger range of $\lambda$ for which they achieve a low suboptimality.
Interestingly, for the absolute loss, the difference is barely noticeable.
We also observe that \SPLplus/ often admits a wider basin of good settings for $\lambda$ as compared to SGM and even SPL. 
Moreover, $\lambda =1$ is often in the set of good parameter choices for \SPLplus/. 
This suggest that our scaling of $
\lambda \ell_0$ and $\nicefrac{\lambda}{\theta_0}$, as motivated by balancing units, lead to a more stable and easy to tune method by choosing $\lambda$ around $1$.
In \Figref{fig:iters-vs-lambda}, we perform the sensitivity analysis 
under a fixed accuracy target $\tilde F(\theta,\alpha) - \tilde F^* \leq \epsilon$,
and draw similar stability conclusions.

\begin{figure*}
	\begin{center}
		\includegraphics[scale=0.8]{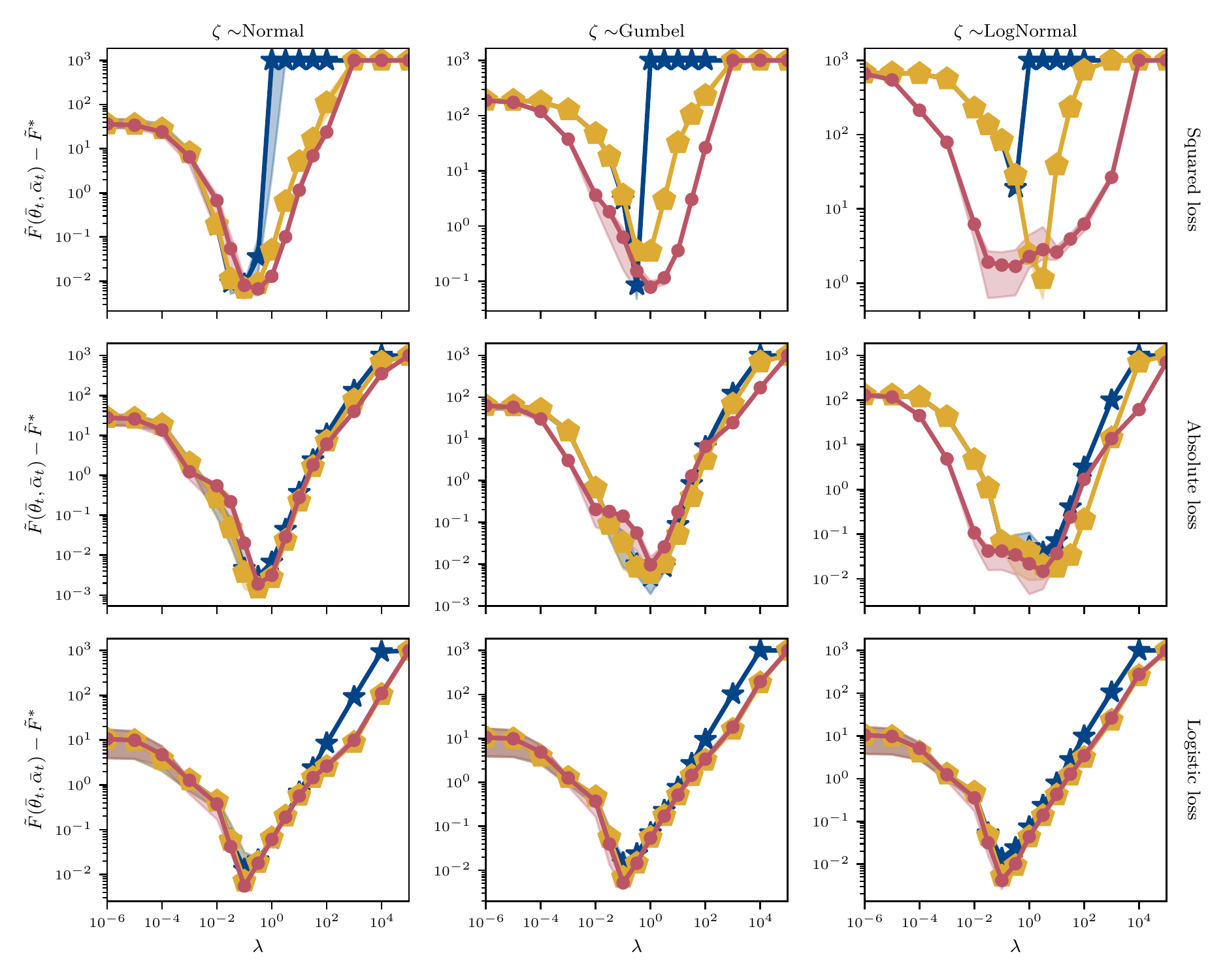}
		\includegraphics{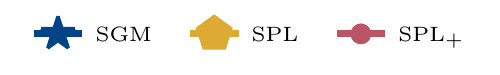}
	\end{center}
	\caption{Sensitivity of final suboptimality to step size choices under a fixed $T=10^5$ budget. 
	The first two rows are regression tasks under the $\ell_1$ and $\ell_2$ losses, 
	while the third row correspond to a binary classification task under the logistic loss. 
	The columns correspond to different noise distributions in the data generation that controls the difficulty of the problem. 
	}
	\label{fig:subopt-vs-lambda}
\end{figure*}

\subsection{Real data}
\label{sec:exp-real}
Finally, we present the same experiment on four real datasets:
\texttt{YearPredictionMSD}, \texttt{E2006-tfidf}, (binary) \texttt{mushrooms} and (binary) \texttt{Covertype},
all from the LIBSVM repository \citep{chang2011libsvm}.
Similar to the synthetic experiments, we set $\beta=0.95$ and compute $\theta^*$ and $\alpha^*$ using L-BFGS.
The objective is now by default the empirical CVaR in~\labelcref{eq:empirical-cvar}
since $P$ is the empirical distribution,
\begin{align*}
	F_\beta(\theta,\alpha) = \alpha + \frac{1}{1-\beta}\frac{1}{n}\sum_{i=1}^n \max\braces{\ell(\theta;z_i)-\alpha,\,0}
\end{align*}
where $N$ is the number of examples in the training split. The loss function 
$\ell(\cdot;z_i)$ is the squared loss for \texttt{YearPredictionMSD} and \texttt{E2006-tfidf}, and logistic loss for 
\texttt{mushrooms} and \texttt{Covertype}.
For the comparison between SGM, SPL, and \SPLplus/, we run the methods for $200N$ iterations 
(except on \texttt{E2006-tfidf} where we only run for $10N$ iterations due to its size). 
All convergence plots are based on the best $\lambda$ at the end of training for each method.

For the least squares problem in \Figref{fig:yearpred} and \Figref{fig:e2006}, we again see that both SPL and \SPLplus/ 
can tolerate a much larger range of step sizes. 
The best $\lambda$ is attained at or near $\lambda=1$
for \SPLplus/, which, although performs slightly worse than SPL with the best selected $\lambda$,
allows us to consistently choose $\lambda=1$ as a default.
For the logistic regression problem in \Figref{fig:covtype} and \Figref{fig:mushrooms}, 
SPL and \SPLplus/ are again similar or better than SGM,
although $\lambda=1$ is no longer close to optimal for SPL and \SPLplus/.

\begin{figure}[h]
	\begin{center}
		\includegraphics[width=0.8\textwidth]{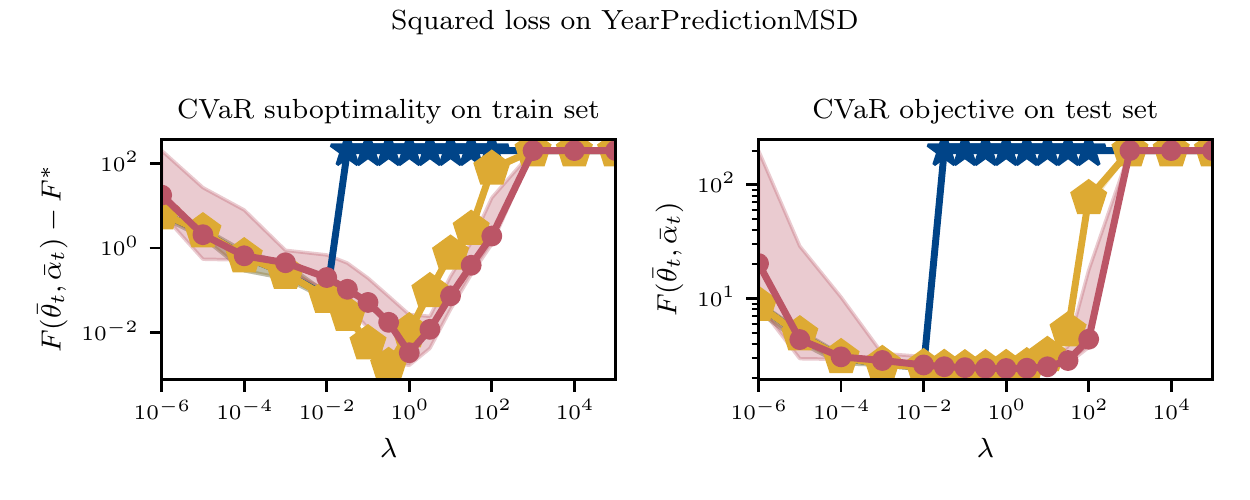}
		\includegraphics[width=0.8\textwidth]{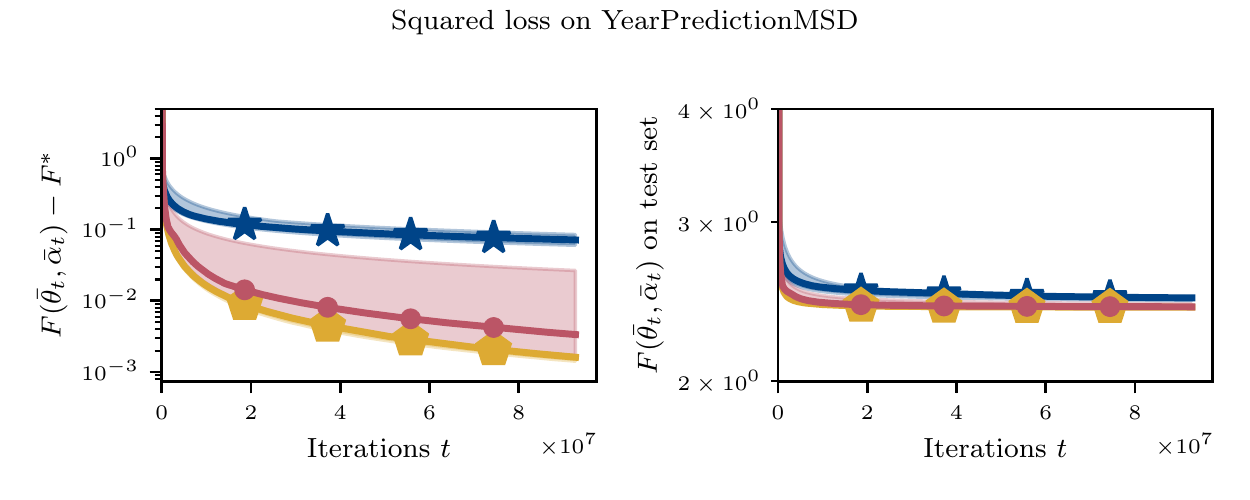}
		\includegraphics[scale=1.2]{figs/legend.pdf}
	\end{center}
	\caption{Sensitivity and convergence plots on the \texttt{YearPredictionMSD} linear regression task (overdetermined).}	\label{fig:yearpred}
\end{figure}

\begin{figure}[h]
	\begin{center}
		\includegraphics[width=0.8\textwidth]{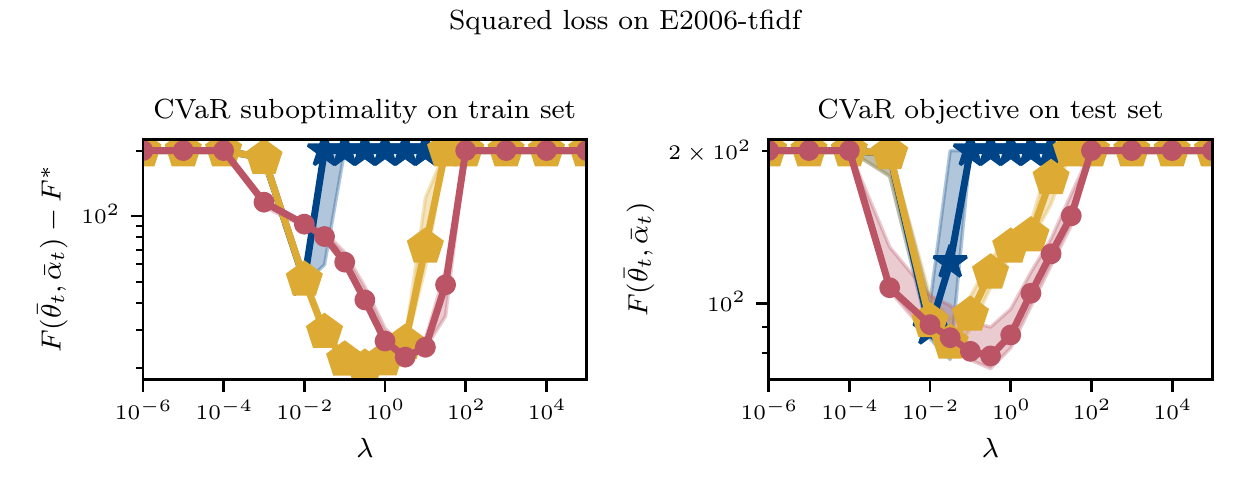}
		\includegraphics[width=0.8\textwidth]{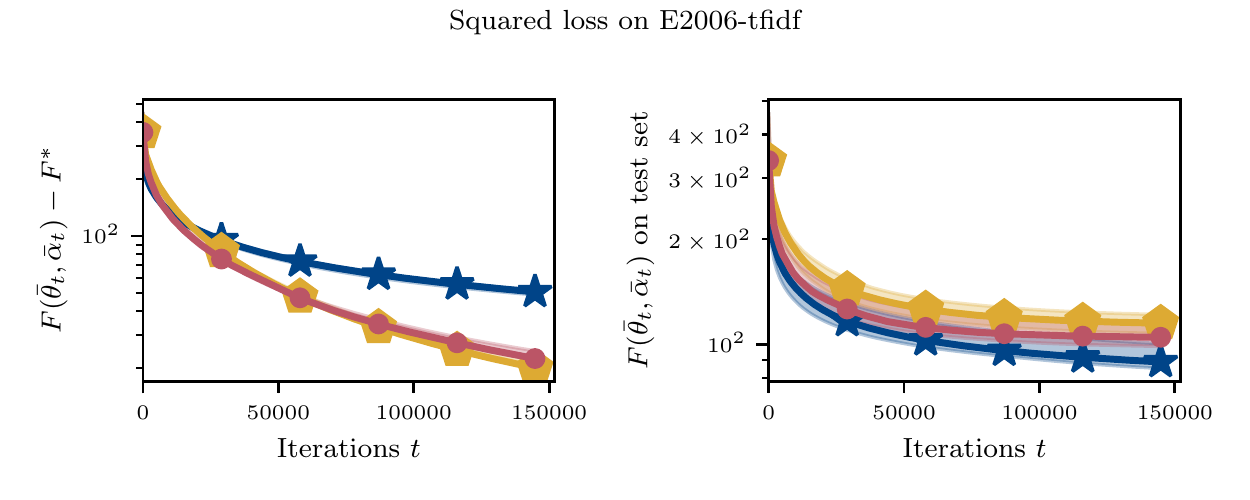}
		\includegraphics[scale=1.2]{figs/legend.pdf}
	\end{center}
	\caption{Sensitivity and convergence plots on the \texttt{E2006-tfidf} linear regression task (underdetermined).}	\label{fig:e2006}
\end{figure}

\begin{figure}[h]
	\begin{center}
		\includegraphics[width=0.8\textwidth ]{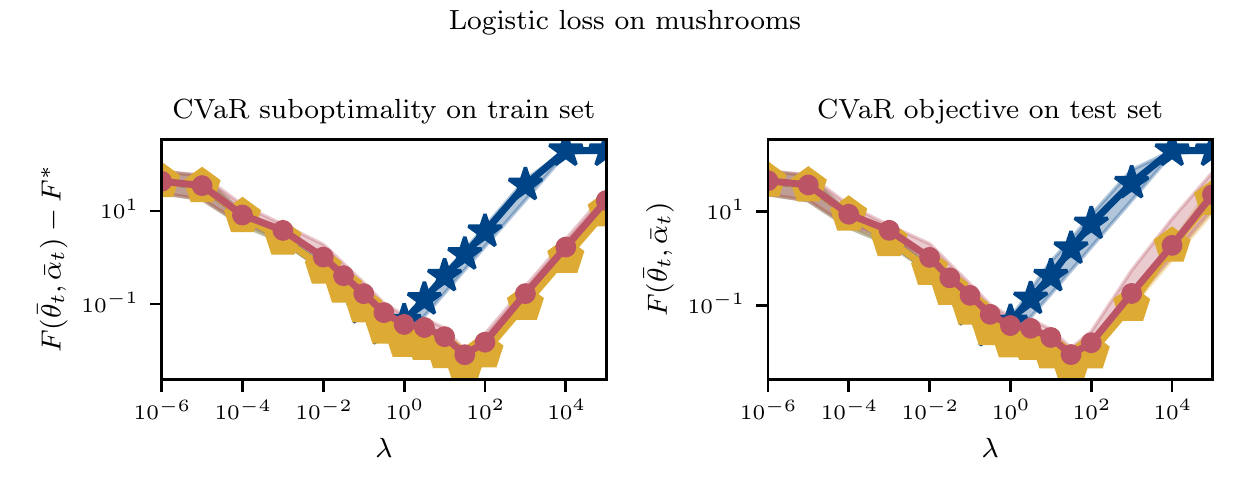}
		\includegraphics[width=\textwidth]{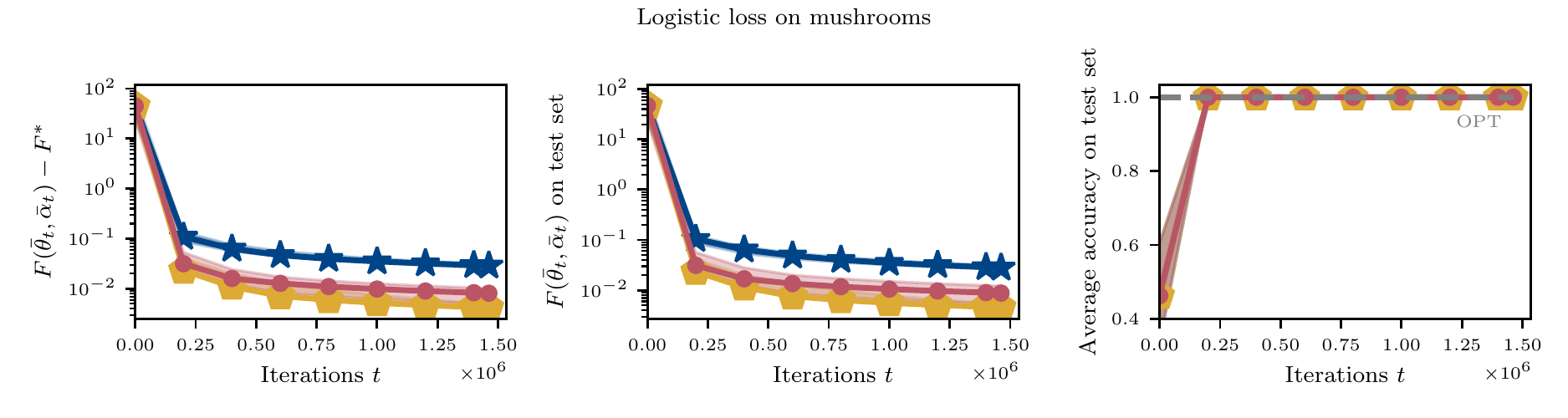}
		\includegraphics[scale=1.2]{figs/legend.pdf}
	\end{center}
	\caption{Sensitivity and convergence plots on the \texttt{mushrooms} binary classification task.
	The grey dashed line is the average accuracy on the test set achieved by $\theta^*$. }
	\label{fig:mushrooms}
\end{figure}

\begin{figure}[h]
	\begin{center}
		\includegraphics[width=0.8\textwidth ]{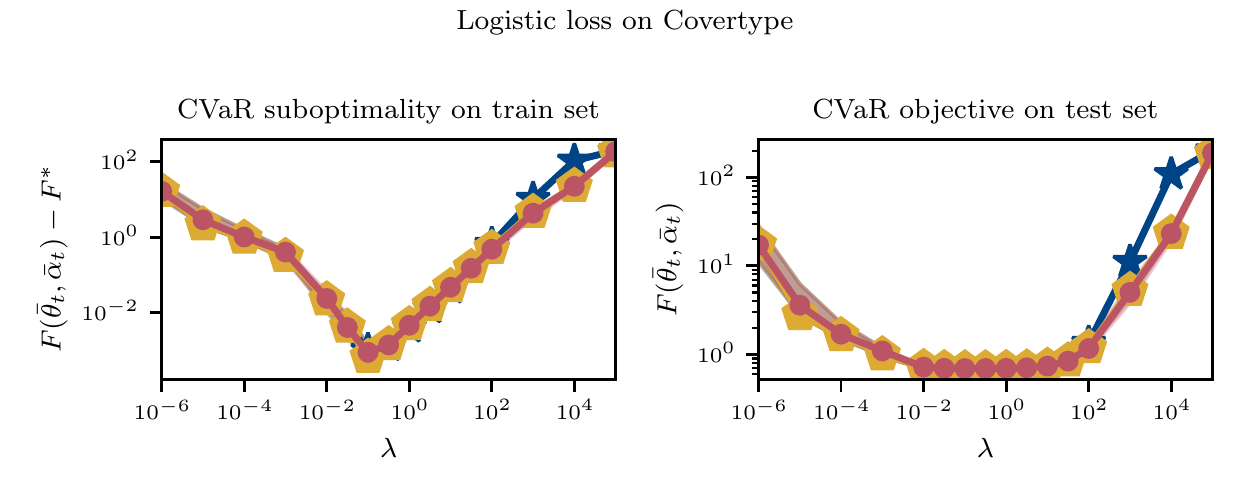}
		\includegraphics[width=\textwidth]{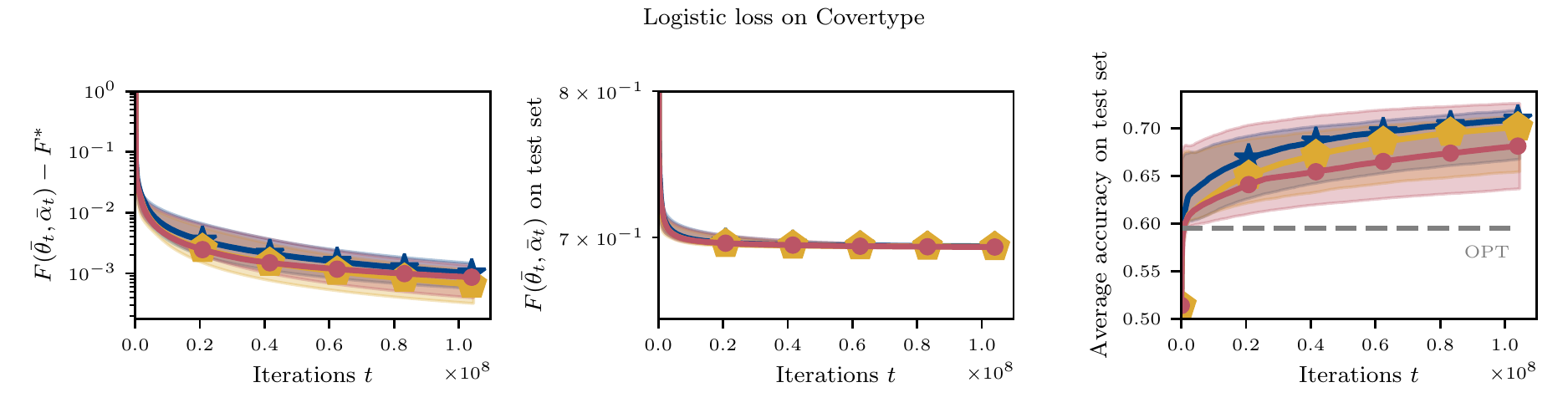}
		\includegraphics[scale=1.2]{figs/legend.pdf}
	\end{center}
	\caption{Sensitivity and convergence plots on the \texttt{Covertype} binary classification task.
	The grey dashed line is the average accuracy on the test set achieved by $\theta^*$. 
	Note that the reported accuracy is averaged across the entire training set, but since SPL and \SPLplus/ reached a lower 
	CVaR objective (rather than the average loss objective), it is reasonable that its average accuracy is lower.
	Furthermore, the optimal accuracy on the test set may seem surprisingly low. 
	The justification behind this is that the objective being minimized is the CVaR objective, while the metric being plotted 
	is the average test accuracy. The CVaR objective puts more emphasis on the top 
 	$1-\beta$ fraction of the examples, and so when the dataset is not linearly separable,
	the average accuracy can be poor. This has also been noted previously in \citet{curi2020adaptive}.
	Unfortunately, computing the accuracy of the examples with the top $1-\beta$ 
	fraction of the losses does not necessarily give us more insight into the accuracy metric. 
	This is due to the possibility that there
	are only a few outliers and the classifier found by minimizing the CVaR is still getting the majority 
	of the examples wrong, just with lower losses.}
	\label{fig:covtype}
\end{figure}

\section{Conclusion and future work}
Our numerical evidence suggests that for the CVaR minimization problem,
while both SGM and SPL can be tuned to achieve similar performance, 
 \SPLplus/  is often the most tolerant to misspecified step sizes.
To further speed up \SPLplus/  and make it more competitive over SGM, in future work we will consider using  non-uniform sampling to bias towards training examples
with higher losses (as in \citet{curi2020adaptive,sagawa2019distributionally}).

Efficient CVaR minimization with a stochastic algorithm opens up the possibility
for new applications in machine learning. 
For instance, we could consider models that trade-off between low average risk
and heavy tails by adding the CVaR objective as a regularizer: $$\min_{\theta\in\R^d} R_{\mathrm{ERM}}(\theta) + \rho R_{\mathrm{CVaR}_\beta}(\theta)$$ 
where $\rho>0$ is a parameter that captures this trade-off. Controlling this trade-off is important as machine learning models are increasingly deployed in safety-critical applications that call for control over the likelihood of failure.
As future work, we also see applications in training neural networks, where CVaR can be used to disincentivize the activations from being saturated too often, and thus help in speeding up training. This would offer an alternative to normalization layers,  such as batchnorm or layernorm.

\clearpage
\acks{We would like to thank Vasileios Charisopoulos and Frederik Künstner
for helpful feedback on an earlier draft. We acknowledge the support of the Natural Sciences and Engineering Research Council of Canada 
(NSERC), Grant No. PGSD3-547276-2020. This work was partially done during S. Y. Meng's internship at the Flatiron Institute.}

\vskip 0.2in
\bibliography{refs}

\newpage
\appendix
\section{\SPLplus/ derivation for CVaR minimization}
\label{sec:app-alg}

Before deriving the updates, we first introduce the following lemma 
based on the truncated model from \citet{asi2019stochastic}.

\begin{lem}[Truncated model] \label{lem:truncated-model}
	Consider the problem 
	\begin{align*}
	  x_{t+1} = \argmin_{x\in\R^n} \; \max\braces{c + \inner{a }{x-x_t},\, 0} + \frac{1}{2\lambda} \normsq{x-x_t}.
	\end{align*}
	for some scalar $c$ and vector $a\in\R^n$. The solution can be
    written in closed form as
	\begin{align*}
	  x_{t+1} = x_t -  \min\braces{\lambda,\, \frac{ \max\braces{c,\,0} }{\normsq{a}}} a
	\end{align*} 
\end{lem}
\begin{proof}
Note that $x_{t+1}$ is the proximal point of the function
\[
    f(x) = h(\dotp{a, x} + b), \quad \text{with} \; \;
    h(z) = \max\braces{z,\, 0}, \; b = c - \dotp{a, x_t}.
\]
centered at $x \equiv x_{t}$. Using~\citet[Theorem 6.15]{beck2017first},
we have
\begin{align}
    \prox_{\lambda f}(x) &= x + \frac{a}{\|a\|^2}\left(
        \prox_{\lambda \norm{a}^2 h}(\dotp{a, x} + b) -
        (\dotp{a, x} + b)
    \right) \notag \\
    &=
    x_t + \frac{a}{\norm{a}^2}\left(
        \prox_{\lambda \norm{a}^2 \max\braces{\cdot,\,0}}(c) - c
    \right)
    \label{eq:composition-prox}
\end{align}
In turn, the max function is the support function of the interval $[0, 1]$.
By~\citet[Theorem 6.46]{beck2017first}, it follows that
\begin{equation}
    \prox_{\lambda \norm{a}^2 \max\braces{\cdot,\,0}}(c) =
    c - \lambda \norm{a}^2 \mathrm{proj}_{[0, 1]}\left(
        \frac{c}{\lambda \norm{a}^2}
    \right).
    \label{eq:max-prox}
\end{equation}
Plugging~\labelcref{eq:max-prox} into~\labelcref{eq:composition-prox}, we obtain
\begin{align*}
    \prox_{\lambda f}(x_t) &= x_t -
        \frac{a}{\norm{a}^2} \cdot \lambda \norm{a}^2 \mathrm{proj}_{[0, 1]}\left(\frac{c}{\lambda \norm{a}^2}\right) \\
        &=
    x_t - \lambda a \cdot \mathrm{proj}_{[0, 1]}\left(
        \frac{c}{\lambda \norm{a}^2}
    \right).
\end{align*}
Writing $\mathrm{proj}_{[0, 1]}(v) = \min\braces{\max\braces{v, \,0},\, 1}$ yields the result.
\end{proof}
\closedupdate*
\begin{proof}
We now derive the the \SPLplus/ updates. Recall that for the CVaR objective, using the model $m_{t}^{\mathrm{SPL}}$ in \labelcref{eq:stoch-model-based-update-2},
the stochastic model-based approach solves the following problem in~\Eqref{eq:stoch-model-based-update-2} at each iteration, that is
\begin{align}
    \label{eq:modelbasedtwo}
	\argmin_{\theta,\alpha}\;
	\alpha + \frac{1}{1-\beta} \max\braces{\ell(\theta_t;z) + \inner{ v_t }{\theta-\theta_t} - \alpha, 0}  + \frac{1}{2\lambda_\theta} \normsq{\theta - \theta_t} + \frac{1}{2\lambda_\alpha}(\alpha - \alpha_t)^2  
\end{align}
where $v_t\in\partial\ell(\theta_t;z)$, and we have temporarily dropped the time-dependence on $\lambda_{\alpha,t}$ and $\lambda_{\theta,t}$.
To arrive at the closed form solution, we will re-write~\labelcref{eq:modelbasedtwo} to fit the format of Lemma~\ref{lem:truncated-model}, and then apply the lemma. 
To this end,  we combine the $\alpha$ in front with its regularization term,
\begin{align*}
	\alpha + \frac{1}{2\lambda_\alpha} (\alpha - \alpha_t)^2 &= \alpha + \frac{1}{2\lambda_\alpha}(\alpha^2 - 2\alpha\alpha_t + (\alpha_t)^2) \\
	&= \frac{1}{2\lambda_\alpha}((\alpha - \alpha_t)^2 + 2\lambda_\alpha \alpha ) \\
	&= \frac{1}{2\lambda_\alpha}((\alpha - \alpha_t)^2 + 2\lambda_\alpha \alpha - 2\lambda_\alpha\alpha_t + \lambda_\alpha^2) + \frac{1}{2\lambda_\alpha}( 2\lambda_\alpha\alpha_t -\lambda_\alpha^2)  \\
	&= \frac{1}{2\lambda_\alpha} (( \alpha - \alpha_t )^2 + 2\lambda_\alpha(\alpha - \alpha_t) + \lambda_\alpha^2 ) + \text{Const.}\\
	&= \frac{1}{2\lambda_\alpha}(\alpha + \lambda_\alpha - \alpha_t )^2 + \text{Const.}
  \end{align*}
  
  We now combine it with the regularization on $\theta$
  \begin{align*}
	\underbrace{\frac{1}{2\lambda_\theta}\normsq{\theta-\theta_t} + \alpha + \frac{1}{2\lambda_\alpha} (\alpha - \alpha_t)^2}_{(*)} &= \frac{1}{2\lambda_\theta} \paren{ \normsq{\theta - \theta_t} + \frac{\lambda_\theta}{\lambda_\alpha} (\alpha-\alpha_t + \lambda_\alpha)^2 } + \text{Const.} \\
	&= \frac{1}{2\lambda_\theta}\paren{ \normsq{\theta-\theta_t} + \paren{ \sqrt{\frac{\lambda_\theta}{\lambda_\alpha} } (\alpha-\alpha_t) + \sqrt{\lambda_\theta\lambda_\alpha}}^2 }  + \text{Const. }
  \end{align*}
  Now we define a rescaled variable $\alpha$ and constant $\alpha_t$ as 
  \begin{align}
	\label{eq:rescaled-alpha}
	\hat\alpha = \sqrt{\frac{\lambda_\theta}{\lambda_\alpha}}\alpha  \qquad \text{and} \qquad \hat \alpha_t = \sqrt{\frac{\lambda_\theta}{\lambda_\alpha}} \alpha_t - \sqrt{\lambda_\theta\lambda_\alpha}
  \end{align}
  to arrive at 
  \begin{align*}
	(*) = \frac{1}{2\lambda_\theta} \paren{\normsq{\theta-\theta_t} + (\hat\alpha - \hat\alpha_t)^2 } + \text{Const. }
  \end{align*}
  As a side note: to see that the units argument is appropriate, observe that $\hat\alpha$ now has $\mbox{units}(\theta)$
  since $\lambda_\theta$ has units inversely proportional to $\lambda_\alpha$.
  This lets us concatenate $\alpha$ with $\theta$ for form a new variable vector $x\in\R^{d+1}$ to have the same units overall.
  Now define 
  \begin{align}
	\label{eq:stacked-theta-alpha}
	x = \begin{pmatrix}
	  \theta \\ \hat\alpha 
	\end{pmatrix} \qquad \text{and} 
  \qquad x_t = \begin{pmatrix}
	  \theta_t \\ \hat\alpha_t
	\end{pmatrix}.
  \end{align}
  The linearization inside $\max\braces{\cdot,\,0}$ in \labelcref{eq:modelbasedtwo} can be written as 
  \begin{align}
	\ell(\theta_t;z) + \inner{\nabla\ell(\theta_t;z)}{\theta - \theta_t} - \alpha &= \ell(\theta_t;z) + \inner{\nabla\ell(\theta_t;z)}{\theta - \theta_t} - \sqrt{\frac{\lambda_\theta}{\lambda_\alpha}} \sqrt{\frac{\lambda_\alpha}{\lambda_\theta}} \alpha \label{eq:tempo9hxz848}\\
	&= \ell(\theta_t;z) + \inner{\nabla\ell(\theta_t;z)}{\theta - \theta_t} - \sqrt{\frac{\lambda_\alpha}{\lambda_\theta}} \hat\alpha + \sqrt{\frac{\lambda_\alpha}{\lambda_\theta}} \hat\alpha_t - \sqrt{\frac{\lambda_\alpha}{\lambda_\theta}} \hat\alpha_t \nonumber \\
	&= \ell(\theta_t;z) - \sqrt{\frac{\lambda_\alpha}{\lambda_\theta}} \hat\alpha_t + \begin{pmatrix}
	  \nabla\ell(\theta_t;z) & - \sqrt{\frac{\lambda_\alpha}{\lambda_\theta}}
	\end{pmatrix} \begin{pmatrix}
	  \theta - \theta_t \\
	  \hat\alpha - \hat\alpha_t
	\end{pmatrix}, \nonumber
  \end{align}
  and so minimizing the model $m_t$ is then equivalent to minimizing the following model $\hat m_t$
  \begin{align*}
	\min_{x\in\R^{d+1}} 
	\max\braces{ c + \inner{a}{x-x_t}, \, 0 } + \frac{1}{2\lambda_\theta} \normsq{x-x_t}
  \end{align*}
  up to constants, where 
  \begin{align*}
	c = \frac{1}{1-\beta} \paren{ \ell(\theta_t;z) - \sqrt{\frac{\lambda_\alpha}{\lambda_\theta}}\hat\alpha_t} \qquad \text{and} \qquad a = \frac{1}{1-\beta}\begin{pmatrix}
	  \nabla\ell(\theta_t;z) \\
	  -\sqrt{\frac{\lambda_\alpha}{\lambda_\theta}}.
	\end{pmatrix}.
  \end{align*} 
  From \cref{lem:truncated-model}, the update is given by 
  $$
  x^* = x_t - \eta \cdot a 
  $$
  where step size is given by 
  \begin{align}
    \label{eq:eta-expression}
    \eta \coloneqq \min\braces{ \lambda_\theta,\, \frac{ \max\braces{c, \,0} }{\normsq{a}} }
  \end{align}
   Plugging in $a,c$ into $\eta$ gives
\begin{align*}
    \theta_{t+1} & = \theta_t - \min\braces{ \lambda_\theta,\, \frac{1}{1-\beta} \frac{ \max\braces{\ell(\theta_t;z) - \sqrt{\frac{\lambda_\alpha}{\lambda_\theta}}\hat\alpha_t, \,0} }{\frac{1}{(1-\beta)^2} (\normsq{ \nabla\ell(\theta_t;z)} +\frac{\lambda_{\alpha}}{\lambda_{\theta}} ) } }  \frac{\nabla\ell(\theta_t;z)}{1-\beta} \\
    \hat\alpha_{t+1} & = \hat\alpha_{t} + \frac{1}{1-\beta} \sqrt{\frac{\lambda_{\alpha}}{\lambda_{\theta}} } \min\braces{ \lambda_\theta,\, \frac{1}{1-\beta} \frac{ \max\braces{\ell(\theta_t;z) - \sqrt{\frac{\lambda_\alpha}{\lambda_\theta}}\hat\alpha_t, \,0} }{\frac{1}{(1-\beta)^2} (\normsq{ \nabla\ell(\theta_t;z)} +\frac{\lambda_{\alpha}}{\lambda_{\theta}} ) } }  .
\end{align*}
  Finally, substituting back using~\labelcref{eq:rescaled-alpha}, that is $\hat \alpha_{t+1} = \sqrt{\frac{\lambda_\theta}{\lambda_\alpha}} \alpha_{t+1}$ and  $\hat \alpha_{t} = \sqrt{\frac{\lambda_\theta}{\lambda_\alpha}} \alpha_{t} -\sqrt{\lambda_{\theta} \lambda_{\alpha}}$ and simplifying gives~\labelcref{eq:theta-update-SPL-2reg} and~\labelcref{eq:alpha-update-SPL-2reg}.
\end{proof}

\begin{lem}
   Each \SPLplus/ update in \Algref{alg:spl-cvar-two-reg} is equivalent to the updates given by~\Eqref{eq:alpha-update-SPL-2reg} and~\Eqref{eq:theta-update-SPL-2reg}.
\end{lem}
\begin{proof}
  We can enumerate all the cases:
  \begin{enumerate}
    \item If $c<0$, which implies checking for 
    \begin{align*}
      \ell(\theta_t;z) < \sqrt{\frac{ \lambda_\alpha }{ \lambda_\theta }} \hat \alpha_t = \sqrt{\frac{ \lambda_\alpha }{ \lambda_\theta }} \paren{ \sqrt{\frac{ \lambda_\theta }{ \lambda_\alpha }} \alpha_t - \sqrt{\lambda_\theta\lambda_\alpha}} = \alpha_t - \lambda_\alpha
    \end{align*} 
    then from \labelcref{eq:eta-expression} $\eta=0$, and the updates are 
    \begin{align*}
      \theta_{t+1} &= \theta^* = \theta_t \\
      \alpha_{t+1} &= \hat\alpha^* = \hat\alpha_t  
    \end{align*}
    Multiplying the second equation by $\sqrt{ \frac{\lambda_\alpha}{ \lambda_\theta }}$ on both sides, we get 
    \begin{align*}
      \sqrt{ \frac{\lambda_\alpha}{ \lambda_\theta }} \sqrt{ \frac{\lambda_\theta}{ \lambda_\alpha }}  \alpha^* &= \sqrt{ \frac{\lambda_\alpha}{ \lambda_\theta }} \paren{ \sqrt{ \frac{\lambda_\theta}{ \lambda_\alpha }}\alpha_t - \sqrt{\lambda_\theta\lambda_\alpha} } \\
      \alpha_{t+1} &= \alpha_t - \lambda_\alpha.
    \end{align*}
    \item If $c > \lambda_\theta\normsq{a} \quad (>0)$, which implies checking for the condition
    \begin{align*}
      \frac{1}{1-\beta} \paren{ \ell(\theta_t;z) - \sqrt{ \frac{\lambda_\alpha}{\lambda_\theta} } \hat\alpha_t } &> \frac{1}{(1-\beta)^2} \paren{ \lambda_\theta \normsq{ \nabla\ell(\theta_t;z) } + \lambda_\alpha } \\
      \ell(\theta_t;z) - \alpha_t \lambda_\alpha &> \frac{1}{1-\beta} \paren{ \lambda_\theta\normsq{ \nabla \ell(\theta_t;z) } + \lambda_\alpha  }.
    \end{align*}
    Then $\eta = \lambda_\theta$, and the updates reduce to 
    \begin{align*}
      \theta_{t+1} &= \theta_t -  \lambda_\theta \frac{1}{1-\beta} \nabla \ell(\theta_t;z) \\
      \alpha_{t+1} &= \hat\alpha^* = \hat\alpha_t - \lambda_\theta \frac{1}{1-\beta} \paren{ -\sqrt{ \frac{\lambda_\alpha}{\lambda_\theta} }  } \\
      &= \alpha_t - \lambda_\alpha + \frac{1}{1-\beta}\lambda_\alpha.
    \end{align*}
    \item Otherwise it must be the case that $0 < \frac{c}{\normsq{a}} < \lambda_\theta$, so $\eta = \frac{c}{\normsq{a}}$, and the updates are given by
    \begin{align*}
      \begin{pmatrix}
        \theta_{t+1} \\ \alpha_{t+1}
      \end{pmatrix} = 
      \begin{pmatrix}
        \theta^* \\ \hat\alpha^* 
      \end{pmatrix} &= \begin{pmatrix}
        \theta_t \\ \hat\alpha_t
      \end{pmatrix} - \frac{c}{\normsq{a}} \cdot a \\
      &= \begin{pmatrix}
        \theta_t \\ \hat\alpha_t
      \end{pmatrix} 
      - \frac{ \ell(\theta_t;z) - \sqrt{ \frac{\lambda_\alpha}{\lambda_\theta} } \hat\alpha_t }{ \normsq{ \nabla\ell(\theta_t;z) } + \frac{\lambda_\alpha}{\lambda_\theta} } 
      \cdot \begin{pmatrix}
      \nabla\ell(\theta_t;z)  \\ -\sqrt{ \frac{\lambda_\alpha }{ \lambda_\theta } }
      \end{pmatrix} \\
      &= \begin{pmatrix}
        \theta_t \\ \hat\alpha_t
      \end{pmatrix} 
      - \underbrace{\frac{ \ell(\theta_t;z) - \alpha_t + \lambda_\alpha }{ \lambda_\theta \normsq{ \nabla\ell(\theta_t;z) } + \lambda_\alpha} 
      }_{ \eqqcolon \nu } \lambda_\theta \cdot \begin{pmatrix}
      \nabla\ell(\theta_t;z)  \\ -\sqrt{ \frac{\lambda_\alpha }{ \lambda_\theta } }
      \end{pmatrix} 
    \end{align*}
    Converting $\hat\alpha_t$ to $\alpha_t$ and $\hat\alpha^*$ to $\alpha^*$, we get that the updates are 
    \begin{align*}
      \theta_{t+1} &= \theta_t - \lambda_\theta \nu \nabla\ell(\theta_t;z) \\
      \alpha_{t+1} &= \alpha_t - \lambda_\alpha + \lambda_\alpha \nu.
    \end{align*}
  \end{enumerate}
Note that the regularization parameters $\lambda_\theta$ and $\lambda_\alpha$ can both be written in a time-dependent form as 
$\lambda_{\theta,t}$ and $\lambda_{\alpha,t}$. This concludes our derivation for the updates of \SPLplus/ given in \Algref{alg:spl-cvar-two-reg}.
As a comparison, we also include the closed-form updates for SGM applied to CVaR 
minimization in \Algref{alg:sgd-cvar}.

\end{proof}

\section{Proof of \cref{thm:general-convex-rate}}
\label{sec:proofs}

\convergenceconvex*

\begin{proof}
For the proof, we 
Recall the model-based update \labelcref{eq:modelbasedtwo} (restated here)
\begin{align*}
	\argmin_{\theta,\alpha}\; \alpha + \frac{1}{1-\beta} & \max\braces{\ell(\theta_t;z) + \inner{ v_t }{\theta-\theta_t} - \alpha, 0} 
	+ \frac{1}{2\lambda_\theta} \normsq{\theta - \theta_t} + \frac{1}{2\lambda_\alpha}(\alpha - \alpha_t)^2  \nonumber
\end{align*}
is the subproblem we solve to obtain the updates with separate regularization. 
Again, we have temporarily dropped the time-dependency on $\lambda_{\alpha,t}$ 
and $\lambda_{\theta,t}$. Since the set of minimizers is the same if we scale the entire expression 
by $\sqrt{\frac{\lambda_\theta}{\lambda_\alpha}}$, we have equivalently
\begin{align}
	\label{eq:rescaled-subprob-for-convergence}
  \argmin_{\theta,\alpha}\; & \sqrt{\frac{\lambda_\theta}{\lambda_\alpha}} \alpha + \frac{1}{1-\beta}\max\braces{ \sqrt{\frac{\lambda_\theta}{\lambda_\alpha} }\paren{ \ell(\theta_t;z) + \inner{v_t}{\theta-\theta_t} } - \sqrt{\frac{\lambda_\theta}{\lambda_\alpha}}\alpha ,\,0} \nonumber \\
  & \phantom{\qquad\qquad\qquad\qquad} + \frac{1}{2\sqrt{\lambda_\theta\lambda_\alpha}}\normsq{\theta-\theta_t} + \frac{1}{2\lambda_\alpha} \sqrt{\frac{\lambda_\theta}{\lambda_\alpha}} \frac{\lambda_\alpha}{\lambda_\theta} \paren{ \sqrt{\frac{\lambda_\theta}{\lambda_\alpha} }\alpha - \sqrt{ \frac{\lambda_\theta}{\lambda_\alpha}} \alpha_t  }^2
\end{align}
Let $\hat\alpha \coloneqq \sqrt{\frac{\lambda_\theta}{\lambda_\alpha}}\alpha $ and $\hat\alpha_t\coloneqq \sqrt{\frac{\lambda_\theta}{\lambda_\alpha} }\alpha_t $
Note that this is a simpler definition of $\hat\alpha_t$ than what we used in the derivation of the updates,
since we no longer have to absorb the leading $\alpha$ into the regularization.
The subproblem \labelcref{eq:rescaled-subprob-for-convergence} can be solved in term of the variables $\theta$ and $\hat\alpha$,
and the scaled linearization
\begin{align}
  \label{eq:scaled-subproblem}
  \argmin_{\theta,\hat\alpha} \hat\alpha + \frac{1}{1-\beta}\max\braces{ \paren{ \hat \ell(\theta_t;z) + \inner{\hat v_t}{\theta-\theta_t} } - \hat\alpha,\, 0 } + \frac{1}{2\sqrt{\lambda_\theta\lambda_\alpha}} \paren{ \normsq{\theta-\theta_t} + (\hat\alpha - \hat\alpha_t)^2  }
\end{align}
where $\hat\ell(\theta_t;z) \coloneqq \sqrt{ \frac{\lambda_\theta }{\lambda_\alpha} } \ell(\theta_t;z)$, and its scaled subgradient is 
$\hat v_t \coloneqq \sqrt{ \frac{\lambda_\theta }{\lambda_\alpha} } v_t$. 
Now define the scaled CVaR objective to be 
\begin{align}
  \label{eq:scaled-variational-cvar}
  \hat F_\beta(\theta,\hat\alpha) = \hat\alpha + \frac{1}{1-\beta}\E_{z\sim P} \brackets{ \max\braces{ \hat \ell(\theta;z) - \hat\alpha,\, 0 } }
\end{align}
and the updates in the scaled subproblem \labelcref{eq:scaled-subproblem} gives the 
\SPLplus/ method for solving this scaled CVaR problem.  By Lemma~\ref{lem:modelassump} we have that the assumptions required to invoke Theorem 4.4 in \citet{davis2019stochastic}
now hold. In particular  since $\ell(\theta;z)$ is $M(z)$--Lipschitz we have that  $\hat \ell(\theta;z)$ is $\sqrt{ \frac{\lambda_\theta }{\lambda_\alpha} } M(z)$--Lipschitz. 
We first consider the convergence of \SPLplus/ in terms of the scaled objectives.
Denoting  $\Delta = \norm{x_0-x^*},$ $\hat\lambda \coloneqq \sqrt{\lambda_\theta\lambda_\alpha}$,
$x= (\theta, \hat\alpha)^\top$, $x^* = (\theta^*, \hat\alpha^*)^\top$ a minimizer of $\hat F_\beta$.

Using a constant step size of $\hat\lambda_t = \frac{\hat\lambda}{\sqrt{T+1}}$, from Theorem 4.4 in \citet{davis2019stochastic}
the convergence rate is 
\begin{align}\label{eq:tempzliuhbol8rz}
  \E\brackets{ \hat F_\beta(\bar x_T) - \hat F_\beta( x^* ) } \leq \frac{ \frac{1}{2}\Delta^2 + \mathsf{L}_{\mathrm{\SPLplus/}}^2 \hat\lambda^2 }{\hat\lambda\sqrt{T+1}}.
\end{align}
Finally, multiplying~\labelcref{eq:scaled-variational-cvar} through by $\sqrt{\frac{\lambda_\alpha }{\lambda_\theta}} $ we have that
\begin{align*}
    \sqrt{\frac{\lambda_\alpha }{\lambda_\theta}}   \hat F_\beta(\theta,\hat\alpha) = \alpha + \frac{1}{1-\beta}\E_{z\sim P} \brackets{ \max\braces{  \ell(\theta;z) - \alpha,\, 0 } } = F_\beta(\theta,\alpha).
\end{align*}
Furthermore, multiplying~\labelcref{eq:tempzliuhbol8rz} through by $\sqrt{\frac{\lambda_\alpha }{\lambda_\theta}} $ and substituting back $\hat\lambda \coloneqq \sqrt{\lambda_\theta\lambda_\alpha}$ and 
$${\Delta}^2 = \norm{\theta_0-\theta^*}^2 + (\hat \alpha_0 -\hat \alpha^*)^2 =
\norm{\theta_0-\theta^*}^2 + {\frac{\lambda_\theta}{\lambda_\alpha}}( \alpha_0 - \alpha^*)^2 $$ gives
\begin{align}\label{eq:tempzliuhbol8rz}
  \E\brackets{  F_\beta(\bar x_T) - F_\beta( x^* ) } &\leq \sqrt{\frac{\lambda_\alpha }{\lambda_\theta}}\frac{ \frac{1}{2}\Delta^2 + \mathsf{L}_{\mathrm{\SPLplus/}}^2 \lambda_\theta\lambda_\alpha }{\sqrt{\lambda_\theta\lambda_\alpha}\sqrt{T+1}} \\
  & =  \frac{ \frac{1}{2}\norm{\theta_0-\theta^*}^2 + \frac{1}{2}\frac{\lambda_\theta}{\lambda_\alpha}( \alpha_0 - \alpha^*)^2 + \mathsf{L}_{\mathrm{\SPLplus/}}^2 \lambda_\theta\lambda_\alpha }{\lambda_\theta\sqrt{T+1}} \\
 &= \frac{1}{2}\frac{\norm{\theta_0-\theta^*}^2}{\lambda_\theta\sqrt{T+1}} + \frac{1}{2}\frac{( \alpha_0 - \alpha^*)^2}{\lambda_\alpha\sqrt{T+1}} + \frac{ \mathsf{L}_{\mathrm{\SPLplus/}}^2  \lambda_\alpha }{\sqrt{T+1}} 
\end{align}
which concludes the proof of convergence of \SPLplus/. 
As for the proof of SGM, it only remains to choose $\lambda_{\theta}=\lambda_{\alpha}=\lambda$
\end{proof}

	To apply Theorem 4.4 in \citet{davis2019stochastic}, 
	we must first verify their assumptions (B1)-(B4) hold. 
	We will enumerate these under their following general setup: 
	writing the CVaR objective in~\Eqref{eq:variational-cvar} as 
	\begin{align}
		F_\beta(x) = f(x) + r(x),
	\end{align} 
 where $r(x)=0$ for SGM while $r(x) = \hat\alpha$ for \SPLplus/. 
	In the \SPLplus/ case, we further write $f(x) = \E_z[ h(c(x;z)) ]$
	where $h(\cdot) = \frac{1}{1-\beta}\max\braces{\cdot,\,0}$
	and $c(x;z) = \hat\ell(\theta;z) - \hat\alpha$. Recall that the stochastic one-sided models used are 
	\begin{align}
		\text{SGM} \qquad &f^{\mathrm{SGM}}_t(x;z) = F_\beta(x_t;z) + \inner{g_t}{x-x_t} \qquad \text{where }g_t \in \partial F_\beta(x_t;z), \quad x = (\theta,\alpha)^\top \label{eq:ftxzdefsSGM} \\
		\text{\SPLplus/} \qquad &f^{\mathrm{SPL}}_t(x;z) = h( c(x_t;z) + \inner{u_t}{x-x_t} ) \qquad \text{where } u_t\in\partial c(x_t;z),  \quad x = (\theta,\hat\alpha)^\top \label{eq:ftxzdefsSPL}
	\end{align}
	and the update in \labelcref{eq:stoch-model-based-update} is equivalent to 
	\begin{align}
		x_{t+1} = \argmin_{x\in\R^{d+1}}\, r(x) + f_t(x;z) + \frac{1}{2\lambda_t} \normsq{x-x_t}
	\end{align}
	
The assumptions we need to verify are given in the following Lemma, which are adapted from \citet{davis2019stochastic}.
\begin{lem} \label{lem:modelassump}
Let $\ell(\theta;z)$ be $M(z)$--Lipschitz and convex. Consider the two alternative definitions for $f_t(x;z)$ given in~\labelcref{eq:ftxzdefsSGM} and ~\labelcref{eq:ftxzdefsSPL}. 
We have that the following assumptions hold.
\begin{enumerate}
	\item[(B1)] (\textbf{Sampling}) 
	\textit{It is possible to generate i.i.d. realizations $z_1,z_2,\hdots\sim P$}.
	\item[(B2)] (\textbf{One-sided accuracy}) 
	\textit{ There is an open set $U$ containing $\dom\,r$ and a measurable function $(x,y;z)\mapsto g_x(y;z)$, defined on $U\times U\times\Omega$, satisfying  
	$$\E_z\brackets{ f_t(x_t;z) } = f(x_t) \quad \forall x_t \in U,$$
	and $$\E_z \brackets{f_t(x;z) - f(x) } \leq \frac{\tau}{2}\normsq{x_t-x} \quad \forall x_t,x\in U. $$}
	
	\item[(B3)] (\textbf{Weak-convexity})
	\textit{The function $f_t(x;z)+r(x)$ is $\eta$-weakly convex for all $x\in U$, a.e. $z\in\Omega$.}

	\item[(B4)] (\textbf{Lipschitz property})    
	\textit{There exists a measurable function $L:\Omega\to\R_+$ satisfying $\sqrt{\E_z[L(z)^2]}\leq \mathsf{L}$ and such that
	$$f_t(x_t;z) - f_t(x;z)\leq L(z)\norm{x_t-x} \quad \forall x_t,x\in U \text{ and a.e. } z\sim P,$$}
 where
     \begin{align}
        \mathsf{L}_{\mathrm{SGM}}^2 &= \E_z\brackets{\frac{M(z)^2+1 }{(1-\beta)^2}+ 1  } & \text{for \textbf{SGM} where $f_t(x;z)$ is~\labelcref{eq:ftxzdefsSGM}}, \\
        \mathsf{L}_{\mathrm{\SPLplus/}}^2 &= \E_z\brackets{\frac{ \frac{\lambda_\theta}{\lambda_\alpha}  M(z)^2+1 }{(1-\beta)^2}} & \text{for \textbf{\SPLplus/} where $f_t(x;z)$ is~\labelcref{eq:ftxzdefsSPL}}.
    \end{align}
\end{enumerate}
\end{lem}

\begin{proof}
Assumption (B1) follows trivially from i.i.d. sampling, while (B2) follows from convexity of $\ell(\cdot;z)$ or $\hat\ell(\cdot;z)$, 
giving us $\tau=0$.
Since $r(x)$ is also convex in both methods and both models are convex, 
(B3) holds with $\eta=0$.
 
To prove item (B4) for SGM, where $f_t = f^{\mathrm{SGM}}_t$ is  given in~\labelcref{eq:ftxzdefsSGM},
first note that from~\labelcref{eq:cvar-sampled-subgrad} for $g_t \in \partial F_{\beta}(x_t;z)$ and $u_t \in \partial \ell(\theta_t;z) $ we have that 
\begin{align}
    \norm{g_t}^2 &= \indicator{\ell(\theta_t;z) - \alpha_t \geq 0} \frac{\norm{u_t}^2}{(1-\beta)^2} + \left(1- \frac{\indicator{\ell(\theta_t;z) - \alpha_t \geq 0}}{(1-\beta)} \right)^2 \nonumber \\
    &\leq \indicator{\ell(\theta_t;z) - \alpha_t \geq 0} \frac{M(z)^2}{(1-\beta)^2} + 1- 2\frac{\indicator{\ell(\theta_t;z) - \alpha_t \geq 0}}{(1-\beta)}  +  \frac{\paren{\indicator{\ell(\theta_t;z) - \alpha_t \geq 0}}^2}{(1-\beta)^2} \nonumber \\
    &\leq \frac{M(z)^2}{(1-\beta)^2} +1 + \frac{1}{(1-\beta)^2},
\end{align}
where in the first inequality we used that $\ell(\cdot;z)$ is $M(z)$--Lipschitz to bound $\norm{u_t} \leq M(z)$, and in the second inequality we used that the indicator function $\indicator{\ell(\theta_t;z) - \alpha_t \geq 0}$ is positive and upper bounded by $1$. 
Consequently,
\begin{align}\label{eq:subgradbndSGD}
    \norm{g_t} \leq  \sqrt{1+ \frac{M(z)^2+1}{(1-\beta)^2}  } .
\end{align}
Thus using the above and that $\max\braces{\cdot,\,0}$ is $1$-Lipschitz:
\begin{align}
	f_t(x_t;z) - f_t(y;z) &\leq \norm{g_t} \norm{x_t-y} \nonumber \\
	& \leq   \underbrace{\sqrt{1+ \frac{M(z)^2+1}{(1-\beta)^2}  }}_{=:L(z)}\norm{x_t-y}. \tag{By \labelcref{eq:subgradbndSGD}}
\end{align}
This gives us $\mathsf{L}_{\mathrm{SGM}}^2 = \E_z[L(z)^2]= \E_{z} \brackets{1+ \frac{M(z)^2+1}{(1-\beta)^2}  }$.

For \SPLplus/ and $f_t = f^{\mathrm{SPL}}_t$ defined in~\labelcref{eq:ftxzdefsSPL} we have that
\begin{align*}
	(1-\beta)(f_t(x_t;z) - f_t(y;z)) &=  \max\braces{ \hat\ell(\theta_t;z) - \hat\alpha_t,\,0 } - \max\braces{\hat\ell(\theta_t;z) - \inner{\hat v_t}{\theta-\theta_t}-\hat\alpha  ,\,0}\\
	&\leq \max\braces{\inner{\hat v_t}{\theta-\theta_t} + (\hat\alpha - \hat\alpha_t)  ,\, 0} \tag{$\max\braces{a,0} - \max\braces{b,0} \leq \max\braces{a-b,0}$ }\\
	&= \max\braces{\left\langle \begin{pmatrix}
		\hat v_t \\
		1
	\end{pmatrix},\, \begin{pmatrix}
		\theta - \theta_t \\
		\hat\alpha - \hat\alpha_t
	\end{pmatrix} \right\rangle ,\,0} \\
	&\leq  \norm{ \begin{pmatrix}
		\hat v_t \\
		1
	\end{pmatrix} } \norm{ \begin{pmatrix}
		\theta - \theta_t \\
		\hat\alpha - \hat\alpha_t
	\end{pmatrix} }  \tag{Cauchy-Schwarz} \\
	& \leq \sqrt{ 1 + \normsq{\hat v_t} } \norm{x_t-y} \\
	&\leq \sqrt{1 + \frac{\lambda_\theta }{\lambda_\alpha}  M(z)^2} \norm{x_t-y} \tag{Since $\hat v_t$ is scaled $v_t$}
\end{align*}
Dividing both sides by $(1-\beta)$ gives us 
\begin{align*}
	f_t(x_t;z) - f_t(y;z) &\leq \underbrace{\paren{\frac{\sqrt{ \frac{\lambda_\theta }{\lambda_\alpha} M(z)^2 + 1 }}{1-\beta}   } }_{=: L(z)} \norm{x_t-y}.
\end{align*}
Taking expectation over $z$ yields
\begin{align*}
\mathsf{L}_{\mathrm{\SPLplus/}}^2 = \E_z[L(z)^2] = \E_z \brackets{ \frac{\frac{\lambda_\theta }{\lambda_\alpha} M(z)^2 + 1 }{(1-\beta)^2}   }.
\end{align*}
\end{proof}



\section{Additional experiment results}
\label{sec:app-exp-details}
\Figref{fig:iters-vs-lambda} shows a similar sensitivity analysis to \Figref{fig:subopt-vs-lambda}
in the main text. Instead of the sensitivity of final suboptimality, here we show the 
sensitivity of the minimum number of iterations to reach $\epsilon$-suboptimality $\tilde F(\theta,\alpha) - \tilde F^*\leq \epsilon$.
\begin{figure}[h]
	\begin{center}
		\includegraphics[scale=0.8]{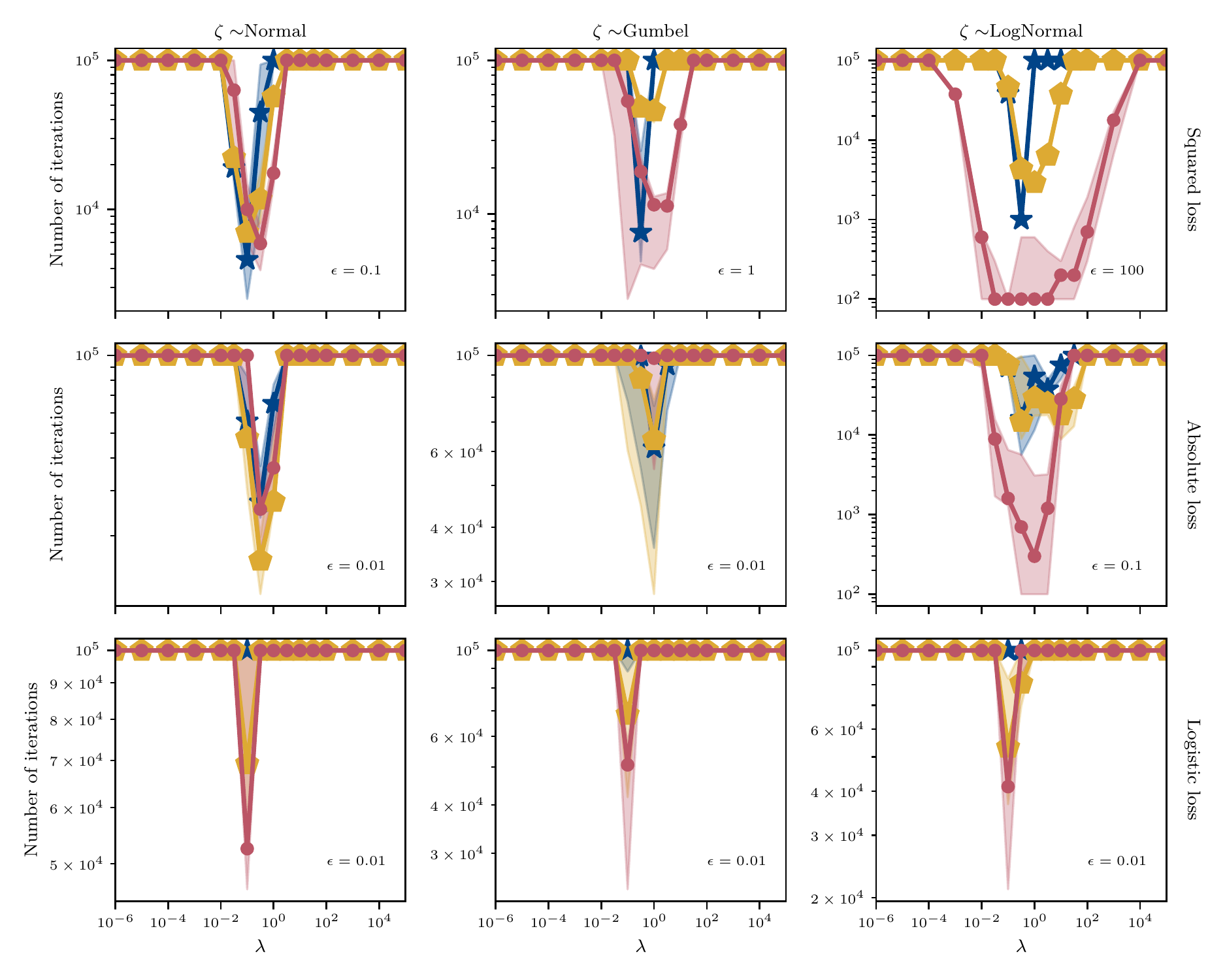}
		\includegraphics[scale=1.2]{figs/legend.pdf}
	\end{center}
	\caption{Sensitivity of minimum number of iterations to achieve $\epsilon$ suboptimality to step size choices. The first two rows are regression tasks under the $\ell_1$ and $\ell_2$ losses, 
	while the third row correspond to a binary classification task under the logistic loss. 
	The columns correspond to different noise distributions in the data generation that controls the difficulty of the problem. }
	\label{fig:iters-vs-lambda}
\end{figure}


\section{Relationship to max loss minimization}
\label{sec:relationship-to-max-loss-minimization}
\begin{lem}
  \label{lemma:max-loss-splplus-equivalence}
  The \SPLplus/ updates in \cref{lem:closed-form-update} minimizes the 
  prox-linear model of the Lagrangian of the max loss objective,
  \begin{align*}
    \min_{\theta\in\R^d} \; f(\theta) = \max_{i=1,\dots,n} \ell(\theta;z_i)
  \end{align*}
  with $\beta = 1-\nicefrac{1}{n}$.
\end{lem}
\begin{proof}
  The equivalent slack formulation to the max loss objective is
\begin{align*}
\min_{s,\theta}&\; s \\
\text{s.t. }\; &\ell(\theta;z_i) \leq s \quad \forall i=1,\dots,n
\end{align*}  
Note that we can add a dummy constraint to have the equivalent problem
\begin{align*}
\min_{s,w}&\; s \\
\text{s.t. }\; &\ell(\theta;z_i) \leq s \quad \forall i=1,\dots,n \\
& 0 \leq 0 \\
& \Updownarrow \\
\min_{s,w}&\; s \\
\text{s.t. }\;& \max\braces{\ell(\theta;z_i)-s,\,0} \leq 0
\end{align*}  
Then the Lagrangian is given by 
\begin{align}
  \label{eq:lagrangian-max-loss}
  \mathcal{L}(s,\theta,\Gamma) = s + \frac{1}{n}\sum_{i=1}^n \Gamma_i \max\braces{\ell(\theta;z_i)-s,\,0}
\end{align}
Note that we have included a $\nicefrac{1}{n}$ scaling for each constraint, which is fine because they are positive and so can be absorbed into the Lagrange multipliers. The dual problem is given by 
\begin{align*}
\max_{\Gamma\in\R^n} &\; \braces{g(\Gamma) = \min_{s,\theta} \mathcal{L}(s,\theta,\Gamma) }\\
\text{s.t. }\; &\Gamma_i \geq 0 \quad \forall i = 1,\dots,n \
\end{align*}
And so given a set of $\Gamma$, we need to  minimize the Lagrangian over $s$ and $\theta$. We can treat this as the \emph{base objective}, 
the basis of our stochastic model construction. At each iteration $t$, we will use the following model
\begin{align*}
m_t(s,\theta) = s + \Gamma_i \max\braces{ \ell(\theta_t;z_i) + \inner{\nabla \ell(\theta_t;z_i)}{\theta-\theta_t} -s,\, 0}
\end{align*}
And then using the stochastic model-based approach, the updates are given by
\begin{align*}
\theta_{t+1}, s_{t+1} = \argmin_{s,\theta} \; m_t(s,\theta) + \frac{1}{2\lambda_\theta}\normsq{\theta-\theta_t} + \frac{1}{2\lambda_s}(s-s_t)^2
\end{align*}
Observe that this corresponds exactly to our \SPLplus/ updates for the CVaR objective. 
Specifically, we can recover that by taking $s=\alpha$ and $\Gamma_i = \frac{1}{1-\beta}$. 
Now let's analyze the KKT conditions to see what $\Gamma$ should be. 
Let $u_i^* = \left.\partial \max\{u, 0\}\right|_{u = \ell(\theta^*; z_i) - s^*}$
\begin{align*}
0 \in \partial_s \mathcal{L}(s^*, \theta^*, \Gamma^*) &=  1-\frac{1}{n}\sum_{i=1}^n \Gamma_i^* u_i^* \iff 1\in \frac{1}{n}\sum_{i=1}^n \Gamma^*_iu^*_i \\ 
0 \in \partial_\theta \mathcal{L}(s^*,\theta^*,\Gamma^*) &= \frac{1}{n}\sum_{i=1}^n \Gamma_i^* u_i^*\nabla \ell(\theta^*;z_i) \\
\max\braces{\ell(\theta^*;z_i)-s^*,\,0} &\leq 0  \qquad \forall i \\
\Gamma_i^* &\geq 0 \qquad \forall i \\
\Gamma_i^*(\max\braces{\ell(\theta^*;z_i)-s^*,\,0} ) &= 0  \qquad \forall i
\end{align*}
First, based on the constraints, we must have $\ell(\theta^*;z_i)\leq s^*$ for all $i$. 
Now suppose $s^*$ does not achieve the max loss, i.e. $\ell(\theta^*;z_i)<s^*$ for all $i$. Then $u_i^*=0$ for all $i$ and the first KKT condition would collapse. 
This means that the set of indices attaining the max $\mathcal{I}=\braces{i=1,\dots,n:\, \ell(\theta^*;z_i)=s^*}$ must be non-empty. We can use this to simplify the first two conditions to 
\begin{align*}
1 &\in \frac{1}{n}\sum_{i\in\mathcal{I}}\Gamma_i^*[0,1] \\
0 &\in \frac{1}{n} \sum_{i\in\mathcal{I}}\Gamma_i^*[0,1]\nabla \ell(\theta^*;z_i)
\end{align*}
which holds for $\Gamma_i^*\geq n$ for $i\in\mathcal{I}$. If we let $\Gamma_i^* = n$ for all $i$, 
then we will need $\beta=1-\nicefrac{1}{n}$ to recover the CVaR objective, which makes sense 
in the max loss minimization problem with $n$ training examples,
so we can take $P=P_n$ the empirical distribution.
\end{proof}


\end{document}